\documentclass[a4paper,oneside,10pt]{article}%
\usepackage{amsmath}
\usepackage{amsfonts}
\usepackage{amssymb}
\usepackage{graphicx}
\usepackage[square,numbers,sort&compress]{natbib}%
\usepackage[colorlinks,linkcolor=blue,anchorcolor=blue,citecolor=blue]{hyperref}
\DeclareMathOperator*{\esssup}{ess\,sup}
\providecommand{\U}[1]{\protect \rule{.1in}{.1in}}

\newtheorem{theorem}{Theorem}[section]

\newtheorem{corollary}[theorem]{Corollary}

\newtheorem{definition}[theorem]{Definition}

\newtheorem{example}[theorem]{Example}

\newtheorem{lemma}[theorem]{Lemma}

\newtheorem{proposition}[theorem]{Proposition}
\newtheorem{remark}[theorem]{Remark}

\newenvironment{proof}[1][Proof]{\noindent \textbf{#1.} }{\  \rule{0.5em}{0.5em}}
\numberwithin{equation}{section}

\begin{document}
\title{On the exit times for SDEs driven by $G$-Brownian motion}
\author{Guomin Liu \thanks{Zhongtai Securities Institute for Financial Studies, Shandong University, Jinan,
 250100, PR China. gmliusdu@163.com.}
\and Shige Peng \thanks{School of Mathematics and Institute for Advanced Research, Shandong University, Jinan,
 250100, PR China.
peng@sdu.edu.cn.  Research
partially supported by the Tian Yuan Projection of the National Natural Sciences Foundation of China (No. 11526205 and No. 11626247) and the 111 Project (No. B12023).} \and Falei Wang\thanks{Zhongtai Securities Institute for Financial  Studies and Institute for Advanced Research, Shandong University,  Jinan, 250100, PR China.
flwang2011@gmail.com. Research partially supported by  the National Natural Science Foundation of China  (No. 11601282) and the Natural Science Foundation of  Shandong Province (No. ZR2016AQ10).}}
\date{}
\maketitle

\textbf{Abstract}. This paper is devoted to studying   the properties of the exit times for stochastic differential equations driven by $G$-Brownian motion ($G$-SDEs).
In particular, we prove that the exit times of $G$-SDEs has the quasi-continuity property.
As an application, we  give a probabilistic representation for a large class of fully nonlinear elliptic equations with Dirichlet boundary.


\addcontentsline{toc}{section}{\hspace*{1.8em}Abstract}

\section{Introduction}
The nonlinear BSDE theory formulated by Pardoux and  Peng \cite{PP1} has many applications  in practice and theory, which range from   economics (see e.g. El Karoui, Peng and Quenez \cite{KPQ}) to PDEs (see e.g. Pardoux and Peng \cite{PP2}, Peng \cite{Peng1}).
Based on BSDE, Peng \cite{Peng2} introduced the nonlinear $g$-expectation theory as a nontrivial generalization of classical linear expectations.
Indeed, the  $g$-expectation  is described by a class of equivalent probability measures. In sprit of this property, Chen and Epstein \cite{CE} studied the stochastic differential recursive utility.

However, many economic and financial problems involve model uncertainty which is characterized
by a family of non-dominated probability measures. Motivated by these questions, Peng \cite{P3,P4,P7} introduced a nonlinear expectation, called $G$-expectation,
which can be regarded as the upper expectation of a specific family of non-dominated probability measures.
Under this framework, the corresponding nonlinear Brownian motion called $G$-Brownian motion is   established.
 Briefly speaking,  $G$-Brownian motion is a continuous process with independent and
stationary increments under  $G$-expectation. Moreover, the stochastic calculus with respect to (symmetric)  $G$-Brownian motion,
forward and backward stochastic differential equations driven by $G$-Brownian motion ($G$-SDEs and $G$-BSDEs in short)  are also obtained, see also Gao \cite{Gao},  Hu, Ji, Peng and Song \cite{HJPS1}.

As is well-known, according to Lusin's theorem, the random variables on the classical probability space are
quasi-continuous (see Section 2 for the definition). But this is no longer true for  the $G$-expectation framework since
 the elements in the probability family that represents $G$-expectation are mutual singular.
 So an important problem for the $G$-expectation theory is the quasi-continuity property  of random variables,
 especially for stopping times which play a major role in classical stochastic analysis but tends to have more discontinuity.

The  purpose of this paper is to study the properties of exit times for $G$-SDEs, among which the most important one is   that, under mild conditions, the exit times of $G$-SDEs have the quasi-continuity property,
so that it belongs to the proper nonlinear $G$-expectation space.
Here, the corresponding $G$-SDEs  are given by
\begin{equation}
dX^{x}_{t}=b(X_{t}^{x})ds+\sum_{i,j=1}^dh_{ij}(X_{t}^{x})d\langle B^i,B^j\rangle_t+\sum_{j=1}^{d}\sigma_j(X_{t}^{x})dB^j_t,\  X_{0}^{x}=x; \ \ \ t\geq 0.
\end{equation}
Different from the usual  case of symmetric $G$-Brownian motion that involves only volatility uncertainty,
in the above equation $B$ is a generalized $G$-Brownian motion, which   has both mean and volatility uncertainty.
Thus we need to study the corresponding stochastic calculus theory first, and  one can refer to \cite{GP,GPP,Nu} for  related discussions.
Next we consider the exit time of $G$-SDEs from an open set $Q$
$$
{\tau}_Q^x:=\inf\{t\geq 0:X^x_t(\omega)\in Q^c\}.
$$
Since we cannot expect the $G$-SDEs have  sufficient continuity  with respect to $\omega$, we introduce
an alternative approach of considering the image space of $G$-SDEs to study  the properties of ${\tau}_Q^x$.
We also utilize the weakly compact method from \cite{Song1}, where the quasi-continuity property of hitting times for symmetric $G$-martingales was considered, and the   strong Markov property of $G$-SDEs from \cite{HJL}.
These properties of exit times  may play an important role for  the applications of $G$-SDEs in many fields   involving a stopping rule.

The well-known  Feynman-Kac formula tells us that stochastic differential equations driven by linear Brownian motion (SDEs)  provide a probabilistic representation for a   class of linear PDEs (with Dirichlet boundary), see, e.g., \cite{Fre}.  With the help of $G$-BSDEs, in  \cite{P7,HJPS2} the authors obtain a
stochastic representation for fully nonlinear parabolic  PDEs in $\mathbb{R}^n$.
Inspired by these results, as an application of our results on the exit times,
 we state a probabilistic interpretation for  a large class of fully nonlinear  elliptic PDEs with Dirichlet boundary via $G$-SDEs.

We also note that   Lions and   Menaldi \cite{LM} (see also  Buckdahn and Nie \cite{BN}) gave a representation for a class of fully nonlinear elliptic equations with Dirichlet boundary via the stochastic control theory under the linear expectation framework. In their construction, every admissible control corresponds to   a trajectory of SDEs.
Compared with the aforementioned results,  the trajectories in our representation are universal defined for all   probability measures.
Moreover,  we prove that the induced probability measures of $G$-SDEs are weakly compact, and hence the supremum of the upper expectation representation can be   realized (Corollary \ref{supremum realization}).
This kind of properties can be applied to the study of first-order  differentiation of the viscosity solutions of fully nonlinear PDEs  (see \cite{HPS,Song2}), which is also our future work.

The paper is organized as follows.
 In Section 2, we present some preliminaries
for nonlinear expectation theory and related space of random variables.
 In Section 3, we give the stochastic integral and differential equations with respect to generalized $G$-Brownian motion.
Section 4 is devoted to the research of the properties of exit times for $G$-SDEs. In Section 5, we provide the probabilistic representation  for fully nonlinear elliptic equations with Dirichlet boundary.
\section{Preliminaries}
The main purpose of this section is to recall some preliminary
results about the upper expectation and the corresponding capacity theory. More
details can be found in  \cite{DHP}.

For  each Euclidian space, we  denote by $\langle\cdot,\cdot\rangle$  and  $|\cdot|$
 its scalar product and the associated norm, respectively.
Let $\Omega_d:=C([0,\infty);\mathbb{R}^d)$ and $B_t(\omega):=\omega(t)$ be the canonical
space and  the canonical mapping equipped with the norm
$$\rho_d(\omega^1,\omega^2):=\sum_{i=1}^\infty\frac{1}{2^i}[(\max_{t\in[0,i]}|\omega^1_t-\omega^2_t|\wedge 1)].$$
The  corresponding natural filtration of $B$ is given by $\mathcal{F}_t:=\sigma \{B_{s}:s\leq t\}$  for $t\geq 0$.

Let $\mathcal{P}$ be a given family of probability measures on $(\Omega_d, \mathcal{B}(\Omega_d))$.
Denote by $\mathcal{L}(\Omega_d,\mathcal{P})$ the space of all $\mathcal{B}(\Omega_d)$-measurable random variables $X$ such that
$E_P[X]$ exists for each $P\in \mathcal{P}$. Next we define the corresponding upper-expectation  by
\begin{equation}
\mathbb{\hat{E}}_{\mathcal{P}}[X]:=\sup_{P\in \mathcal{P}}E_P[X],\ \ \ \ \text{for}\ X \in \mathcal{L}(\Omega_d,\mathcal{P}).
\end{equation}
Then it is easy to check that the triple $(\Omega_d, \mathcal{L}(\Omega_d,\mathcal{P}), \mathbb{\hat{E}}_{\mathcal{P}})$ forms a sublinear expectation space (see Peng \cite{P7}).
In this setting, we can also introduce the notions of identically distribution and independence:
\begin{itemize}
\item[$\cdot$] two $n$-dimensional random vectors  $X=(X_{1},...,X_{n})$ and $Y=(Y_{1},...,Y_{n})$    are called identically distributed, denoted by $X\overset
{d}{=}Y$, if for each $\varphi\in C_{b.Lip}(\mathbb{R}^{n})$,
$
\hat{\mathbb{E}}_{\mathcal{P}}[\varphi(X)]=\hat{\mathbb{E}}_{\mathcal{P}}[\varphi(Y)],
$
\item[$\cdot$] an $m$-dimensional random vector $Y$ is said to be independent
of an $n$-dimensional random vector $X$ if for each
$\varphi\in C_{b.Lip}(\mathbb{R}^{n+m})$,
$
\hat{\mathbb{E}}_{\mathcal{P}}[\varphi(X,Y)]=\hat{\mathbb{E}}_{\mathcal{P}}[\hat{\mathbb{E}}_{\mathcal{P}}%
[\varphi(x,Y)]_{x=X}],
$
\end{itemize}
where $C_{b.Lip}(\mathbb{R}^l)$ is the space of all bounded Lipschitz function defined on $\mathbb{R}^l$, $l\geq1$.
\begin{example}{\upshape
Given two constants $0\leq \underline{\sigma}\leq \bar{\sigma}$.
Suppose   $W$ is a  1-dimensional standard Brownian motion defined  on  Wiener space $(\Omega^0,(\mathcal{F}^0_t)_{t\geq 0},P^0)$, set
\[
\mathcal{P} := \{P_{\theta} : P_{\theta}= P^0\circ X^{-1},\ X_t = \int^t_0 \theta_sdW_s,\  \theta\in\mathcal{A}_{[\underline{\sigma}, \bar{\sigma}]}\},\]
 where $\mathcal{A}_{[\underline{\sigma}, \bar{\sigma}]}$ is the collection of all adapted
 processes taking values in $[\underline{\sigma}, \bar{\sigma}]$. Then on the
 sublinear space $(\Omega_1, \mathcal{L}(\Omega_1,{\mathcal{P}}), \mathbb{\hat{E}}_{\mathcal{P}})$, the  canonical process $B$ is a symmetric
 $G$-Brownian motion ($\hat{\mathbb{E}}_{\mathcal{P}}[B_t]=-\hat{\mathbb{E}}_{\mathcal{P}}[-B_t]=0$) with $G(a)=\frac{1}{2}(\bar{\sigma}^2a^+-\underline{\sigma}^2a^-)$ for each $a\in\mathbb{R}$, see \cite{DHP}.}
\end{example}

Now based on the set of $\mathcal{P}$, we introduce the following capacity, called upper probability,
$$c_{\mathcal{P}}(A):=\sup_{P\in\mathcal{P}} P(A),\ \ A\in \mathcal{B}(\Omega_d).$$
It is obvious that,
\begin{equation}\label{myq1}
c_{\mathcal{P}}(A) =\sup \{ c_{\mathcal{P}}(K):\ K\makebox{ is compact in }\  \mathcal{B}(\Omega_d), \ K\subset A\}, \ \  \forall A\in \mathcal{B}(\Omega_d).
\end{equation}
Then we could establish the language of  ``$\mathcal{P}$-quasi-surely'':
\begin{itemize}
\item[$\cdot$] A set $A\in\mathcal{B}(\Omega_d)$ is called $\mathcal{P}$-polar if $c_{\mathcal{P}}(A)=0$ and  a property is said to  holds ``$\mathcal{P}$-quasi-surely'' ($\mathcal{P}$-q.s.) if it holds outside a polar set. As usual, we do not distinguish between two random variables $X$ and $Y$ if $X=Y$ $\mathcal{P}$-q.s.
\item[$\cdot$] A function $X:\Omega_d \rightarrow \mathbb{R}$ is called
$\mathcal{P}$-quasi-continuous ($\mathcal{P}$-q.c.) if for each $\varepsilon>0$, there exists a closed set
$F$ with $c_{\mathcal{P}}(F^{c})<\varepsilon$ such that $X|_{F}$ is continuous. We say that
$Y:\Omega_d \rightarrow \mathbb{R}$ has a $\mathcal{P}$-quasi-continuous version if
there exists a $\mathcal{P}$-quasi-continuous function $X:\Omega_d \rightarrow \mathbb{R}$ such that $Y=X$ $\mathcal{P}$-q.s.
\end{itemize}
We define the $L^p$-norm of random variables as $||X||_{p,\mathcal{P}}:=(\mathbb{\hat{E}}_{\mathcal{P}}[|X|^p])^{\frac{1}{p}}$ for $p\geq 1$ and set
\[
{L}^{p}(\Omega_d;\mathcal{P}):=\{X\in\mathcal{B}(\Omega_d): ||X||_{{p,\mathcal{P}}}<\infty \}.
\]
Then ${L}^{p}(\Omega_d,\mathcal{P})$ is a Banach space under the norm $||\cdot||_{{p,\mathcal{P}}}$.
Let $C_b(\Omega_d)$ (resp. $B_b(\Omega_d)$) be the space of all bounded, continuous functions (resp. bounded, $\mathcal{B}(\Omega_d)$-measurable functions) on $\Omega_d$. We denote the corresponding completion under norm $||\cdot||_{p,\mathcal{P}}$   by ${L}_{C}^p(\Omega_d,\mathcal{P})$ (${L}_{b}^p(\Omega_d,\mathcal{P})$, resp.).
The following result characterizes the space ${L}_{C}^p(\Omega_d,\mathcal{P}),{L}_{b}^p(\Omega_d,\mathcal{P})$ in  the measurable and integrable sense.
\begin{theorem}[\cite{DHP}]
\label{LG characteriazation theorem}For each $p\geq1$, we have
\[
	L_{b}^{p}(\Omega_d,\mathcal{P})=\{X\in\mathcal{B}(\Omega_d): \lim_{n\rightarrow \infty}\mathbb{\hat{E}}_{\mathcal{P}}[|X|^{p}I_{\{|X|>n\}}]=0\},
	\]
	\[
	L_{C}^{p}(\Omega_d,\mathcal{P})=\{X\in\mathcal{B}(\Omega_d):X\ \text{has a}\ \mathcal{P}\text{-q.c. version,
	}\lim_{n\rightarrow \infty}\mathbb{\hat{E}}_{\mathcal{P}}[|X|^{p}I_{\{|X|>n\}}]=0\}.
	\]
\end{theorem}
Moreover, we have the following monotone convergence results, which are different from the linear case.
\begin{proposition}[\cite{DHP,Song1}]\label{downward convergence proposition} Suppose $X_n$, $n\geq 1$ and $X$ are $\mathcal{B}(\Omega_d)$-measurable.
\begin{description}
\item[(1)]
Assume  $X_n\uparrow X$ q.s. on $\Omega$ and $E_{P}[X_1^-]<\infty$ for all $P\in\mathcal{P}$. Then
$
	\mathbb{\hat{E}}_{\mathcal{P}}[X_n]\uparrow\mathbb{\hat{E}}_{\mathcal{P}}[X].
$
\item[(2)] Assume $\mathcal{P}$ is weakly compact.
\begin{itemize}
\item [(a)]
If $\{X_n\}_{n=1}^\infty$ in ${L}_{C}^{1}(\Omega_d,\mathcal{P})$ satisfies that $X_n\downarrow X$ $\mathcal{P}$-q.s.,  then
	$
	\mathbb{\hat{E}}_{\mathcal{P}}[X_n]\downarrow\mathbb{\hat{E}}_{\mathcal{P}}[X].
	$
\item [(b)] For each closed set $F\in \mathcal{B}(\Omega_d)$,
$
	c_{\mathcal{P}}(F) =\inf \{ c_{\mathcal{P}}(O):\ O\makebox{ open in}\  \mathcal{B}(\Omega_d),\  F\subset O\}.
$
\end{itemize}
\end{description}
\end{proposition}

\begin{remark}\label{remark on tightness guarantee maximum and on closure}
	\upshape{
If ${{\mathcal{P}}}$ is weakly compact, then the maximum  exists for elements of $L^1_C(\Omega_d,{{\mathcal{P}}})$, i.e.,
\begin{equation*}
\mathbb{\hat{E}}_{\mathcal{P}}[X]=\max_{P\in \mathcal{P}}E_P[X],\ \ \ \ \text{for each}\ X\in L^1_C(\Omega_d,{{\mathcal{P}}}).
\end{equation*}
For a family $\mathcal{P}_0$, we denote by ${{\mathcal{{P}}}}$ its closure under weak convergence and it holds that
\begin{equation}\label{2134354544}
\mathbb{\hat{E}}_{\mathcal{P}_0}[X]=\mathbb{\hat{E}}_{{{\mathcal{{P}}}}}[X],\ \ \ \ \text{for each}\ X\in L^1_C(\Omega_d,{{\mathcal{{P}}}}).
\end{equation}
}
\end{remark}

\section{The SDEs driven by  generalized $G$-Brownian motion}
Let $\mathbb{S}(d)$ be the space of all $d\times d$ symmetric matrices.
Consider a fixed sublinear function $G(\cdot,\cdot):\mathbb{S}%
(d)\times\mathbb{R}^d\rightarrow \mathbb{R}$, which is  monotonic in the first variable. Then there exists a bounded and closed set $\Theta\subset \mathbb{R}^{d\times d}\times \mathbb{R}^d$ such that
\begin{equation}\label{generalized G definition}
G(A,p)=\sup_{(\gamma,\mu)\in \Theta}[\frac{1}{2}\langle A,\gamma\gamma^{T}\rangle+\langle p,\mu\rangle].
\end{equation}

In the sequel, we shall introduce an upper expectation on $(\Omega_d, \mathcal{B}(\Omega_d))$ such that the canonical process $B$ is the so-called generalized $G$-Brownian motion. Following the argument of \cite{DHP}, we consider
a linear standard $d$-dimensional Brownian motion $W$ on some
probability space $(\Omega^{0},\mathcal{F}^{0},(\mathcal{F}^{0}_t)_{t\geq 0}, P^{0})$ with
\[
\mathcal{F}^0_{t}:=\sigma \{W_{s},0\leq s\leq t\} \vee \mathcal{N}^{P^{0}},
\]
where $ \mathcal{N}^{P^{0}}$ is the space of all $P^{0}$-null subsets. Denote by $\mathcal{A}^{\Theta}$
the collection of all $\mathcal{F}^0_t$-adapted processes $(\gamma,\mu)$ taking values in $\Theta$
on $\lbrack0,\infty)$. For each fixed
$(\gamma,\mu) \in \mathcal{A}^{\Theta}$ and $0\leq t\leq T<\infty$, we define
\[
B_{T}^{t,\gamma,\mu}:=\int_{t}^{T}\gamma_{s}dW_{s}+\int_{t}^{T}\mu_{s}d{s}.
\]
Then we can obtain a family $\mathcal{P}_0$ of measures:
\begin{equation}\label{Generalized BM P0}
\mathcal{P}_0:=\{P_{\gamma,\mu}:P_{\gamma,\mu}=P^{0}\circ(B_{\cdot}^{0,\gamma,\mu}%
)^{-1},(\gamma,\mu)\in \mathcal{A}^{\Theta}\}.
\end{equation}
which is tight by the Kolmogorov's criterion (see \cite{KS}).
We define its closure under weak convergence as $\mathcal{P}$, which is weakly compact by Prokhorov's theorem. Then we could establish capacity theory corresponding to $\mathcal{P}$ through the results in Section 1.
In the following, for this $\mathcal{P}$,  we will abbreviate $\mathbb{\hat{E}}_{\mathcal{P}}$, $\mathcal{P}\text{-q.s.}$, $c_{\mathcal{P}}$, ${L}_{C}^p(\Omega_d,\mathcal{P})$ as
$\mathbb{\hat{E}},\text{q.s.},c,{L}_{C}^p(\Omega_d)$, etc, for symbol simplicity.

\begin{lemma}\label{integration transformation} For each $ X\in L_C^1(\Omega_d)$,
	\begin{equation}
	\mathbb{\hat{E}}[X]=\sup_{P \in \mathcal{P}_{0}}E_{P}[X]=\sup_{(\gamma,\mu) \in \mathcal{A}^{\Theta}}E_{P^0}[X(B^{0,\gamma,\mu}_{\cdot})].
	\end{equation}
\end{lemma}
\begin{proof}
The proof is immediate from Remark \ref{remark on tightness guarantee maximum and on closure}.
\end{proof}

\begin{remark}\label{L1G contain l,lip r.v.}\upshape{
	 From Theorem \ref{LG characteriazation theorem} and the above lemma, we could get that
$\varphi(B_{t_{1}},B_{t_{2}},\cdots,B_{t_{n}})\in L_C^1(\Omega_d)$,  where $\varphi\in C(\mathbb{R}^{n\times d}) $ is of polynomial growth.}
\end{remark}

The upper expectation $\mathbb{\hat{E}}$  corresponding to $\mathcal{P}$ is called $G$-expectation,
 under which the canonical process $B=(B^1,\cdots,B^d)$ is called  ($d$-dimensional) generalized $G$-Brownian motion, see \cite{P7}. Indeed,
\begin{proposition}\label{properties of generalized G-BM}
Under $\mathbb{\hat{E}}$, the canonical process $B$ is a generalized $G$-Brownian motion, i.e.,
\begin{itemize}
\item [(1)] $B_0=0\ q.s.$ and $\lim_{t\rightarrow 0}\mathbb{\hat{E}}[|B_t|^3]/t=0$;
\item [(2)] $B$ is stationary: $B_{t+s}-B_s\overset{d}{=}B_t$ and has independent increments: $B_{t+s}-B_t$ is independent from $(B_{t_1},\cdots,B_{t_n})$ for any $t_1<\cdots t_n\leq t$ and $s\geq 0$.
\end{itemize}
Moreover,   $\hat{\mathbb{E}}[\langle p,B_t\rangle]\leq G(0,p) t$, which implies $|\hat{\mathbb{E}}[\langle p,B_t\rangle]|\leq [G(0,p)\vee G(0,-p)] t$, for each $p\in \mathbb{R}^d$.
\end{proposition}
\begin{proof}
	The two assertions in  $(1)$ can be easily proved by 	  Lemma \ref{integration transformation} and Remark \ref{L1G contain l,lip r.v.}.
Moreover, by the definition of $G$ and the observation that
	$(\gamma_s,\mu_s)$ take values in $\Theta$, we deduce that for each $p\in \mathbb{R}^d$
	\begin{equation}\label{33334342342}
	\mathbb{\hat{E}}[\langle p,B_t\rangle]
	=\sup_{(\gamma,\mu) \in \mathcal{A}^{\Theta}}E_{P^0}[\langle p,\int_0^t\mu_sds\rangle]
	\leq G(0,p)t,
	\end{equation}
	from which we get   $|\hat{\mathbb{E}}[\langle p,B_t\rangle]|\leq [G(0,p)\vee G(0,-p)] t$.
	
	Now we are going to prove (2). By a similar analysis as in Lemma 43   of \cite{DHP},  we derive that for $\varphi \in C_{b.Lip}(\mathbb{R}^{d})$ and $t,s\geq 0$,
	\begin{align*}
	\sup_{(\gamma,\mu) \in \mathcal{A}^{\Theta}}E_{P^0}[\varphi(B^{0,\gamma,\mu}_{t})]=\sup_{(\gamma,\mu) \in \mathcal{A}^{\Theta}}E_{P^0}[\varphi(B^{s,\gamma,\mu}_{t+s})],
	\end{align*}
	which indicates that
	\begin{equation}\label{elementary BM distribution property}
	\hat{\mathbb{E}}[\varphi(B_{t})]=\hat{\mathbb{E}}[\varphi(B_{t+s}-B_s)].
	\end{equation}
	
Next for each $\varphi \in C_{b.Lip}(\mathbb{R}^{(n+1)\times d})$,  $(\gamma, \mu)\in  \mathcal{A}^{\Theta}$ and   $t,s\geq 0$, taking $\xi:=(B^{0,\gamma,\mu}_{t_1},\cdots,B^{0,\gamma,\mu}_{t_n})$ for $t_1<\cdots t_n\leq t$ and using the argument in Lemma 44  of \cite{DHP},
 we have
	\begin{align*}
	\esssup_{\overline{\gamma},\overline{\mu} \in \mathcal{A}^{\Theta}}%
	E_{P^0}[\varphi((B^{0,\gamma,\mu}_{t_1},\cdots,B^{0,\gamma,\mu}_{t_n}),B^{t,\overline{\gamma},\overline{\mu}}_{t+s})|\mathcal{F}^0_{t}]&
	=(\esssup_{\overline{\gamma},\overline{\mu} \in \mathcal{A}^{\Theta}}%
	 E_{P^0}[\varphi(x,B^{t,\overline{\gamma},\overline{\mu}}_{t+s})|\mathcal{F}^0_{t}])_{x=(B^{0,\gamma,\mu}_{t_1},\cdots,B^{0,\gamma,\mu}_{t_n})}\\&=(\sup_{\overline{\gamma},\overline{\mu} \in \mathcal{A}^{\Theta}}%
	E_{P^0}[\varphi(x,B^{t,\overline{\gamma},\overline{\mu}}_{t+s})])_{x=(B^{0,\gamma,\mu}_{t_1},\cdots,B^{0,\gamma,\mu}_{t_n})}.
	\end{align*}
	Taking firstly expectation $E_{P^0}$ and then   supremum over ${\gamma},{\mu} \in \mathcal{A}^{\Theta}$ to both sides yield that
	$$\sup_{{\gamma},{\mu} \in \mathcal{A}^{\Theta}}%
	E_{P^0}[\varphi((B^{0,\gamma,\mu}_{t_1},\cdots,B^{0,\gamma,\mu}_{t_n}),B^{t,{\gamma},{\mu}}_{t+s})]=\sup_{{\gamma},{\mu} \in \mathcal{A}^{\Theta}}E_{P^0}[(\sup_{\overline{\gamma},\overline{\mu} \in \mathcal{A}^{\Theta}}%
	E_{P^0}[\varphi(x,B^{t,\overline{\gamma},\overline{\mu}}_{t+s})])_{x=(B^{0,\gamma,\mu}_{t_1},\cdots,B^{0,\gamma,\mu}_{t_n})}],
	$$
which is,
	$$
	\hat{\mathbb{E}}[\varphi((B_{t_1},\cdots,B_{t_n}),B_{t+s}-B_t)]=	\hat{\mathbb{E}}[\hat{\mathbb{E}}[\varphi(x,B_{t+s}-B_t)]_{x=(B_{t_1},\cdots,B_{t_n})}].
	$$
The proof is complete.
\end{proof}

Note that when $\Theta$ has only a single point $(\gamma,\mu)$,
 $B$ is  the classical linear Brownian motion with $B_1\sim N(\mu,\gamma\gamma^T)$.
So the generalized $G$-Brownian motion can be regarded as a Brownian motion with mean and covariance uncertainty described by $\Theta$.
 Remark that when $G=G(A)$, the generalized $G$-Brownian motion reduces to symmetric $G$-Brownian motion which has only volatility uncertainty.

 \begin{remark}\upshape{In our article, for the purpose of running a more general PDEs, we use the generalized $G$-Brownian motion.
Most results by now about  symmetric $G$-Brownian motion still hold for generalized $G$-Brownian motion and the proofs are just similar,
so we will give them directly, except that we need to clarify some basic properties of  generalized $G$-Brownian motion and
the construction of related stochastic calculus, which are more sophisticated.
}
\end{remark}

Property $(2)$ in Theorem \ref{properties of generalized G-BM} allows us to define a time-consistent conditional $G$-expectation
 in the following way:
for
{{$X=\varphi(B_{t_{1}},B_{t_{2}
		}-B_{t_{1}},\cdots,B_{t_{n}}-B_{t_{n-1}})$, the conditional expectation at ${t_{j}}$ is
		defined
		by
		\begin{align}
		\hat{\mathbb{E}}_{t_j}[X] :=\psi(B_{t_{1}},B_{t_{2}
		}-B_{t_{1}},\cdots,B_{t_{j}}-B_{t_{j-1}}),\nonumber
		\end{align}
		where}}
\[
\psi(x_{1},\cdots,x_{j})=\hat{\mathbb{E}}[\varphi(x_{1},\cdots
,x_{j},B_{t_{j+1}}-B_{t_{j}},\cdots,B_{t_{n}}-B_{t_{n-1}})].
\]
The  conditional $G$-expectation $\hat{\mathbb{E}}_t[\cdot]$ can be extended continuously to $L_C^1(\Omega_d)$, and can preserve most properties of the linear expectation except the linearity,  see \cite{P7}.

In the remainder of this section, we shall study the stochastic calculus with respect to $B$.
We set $G_1(A):=\sup_{(\gamma,\mu)\in \Theta}\frac{1}{2}\langle A,\gamma\gamma^{T}\rangle=G(A,0)$ and $g(p):=\sup_{(\gamma,\mu)\in \Theta}\langle p,\mu\rangle=G(0,p).$ Then consider two sets:
 $$\Gamma:=\{\gamma\in \mathbb{R}^{d\times d}: \frac{1}{2}\langle A,\gamma\gamma^{T}\rangle\leq G_1(A), \ \text{for each}\ A\in \mathbb{S}%
({d}) \},\ \ \ \ \Sigma:=\{\mu\in \mathbb{R}^d:\langle p,\mu\rangle\leq g(p),\ \text{for each}\ p\in \mathbb{R}^d\}.$$
It is obvious that $\Theta\subset \Gamma\times \Sigma$.
For any $\gamma\in \Gamma$ and $\mu\in \Sigma$, we have
$
0\leq \gamma\gamma^T\leq \overline{\sigma}^2 I_{d\times d }$ and $ |\mu|\leq \beta
$ with $\overline{\sigma}^2:=\sup_{\gamma\in \Gamma}\lambda^{max}[\gamma\gamma^T] $ and $\beta:=\sup_{\mu\in \Sigma}|\mu|,$
where $\lambda^{max}[\gamma\gamma^T]$ is maximal eigenvalue of $\gamma\gamma^T$.

The proof for the following lemma can be found in \cite{STZ} (see also \cite{HP1}).
\begin{lemma}\label{STZ lemma for generalized G}
	For any $P\in\mathcal{P}$, we have
	\begin{equation}E_P[\xi|\mathcal{F}_{t}]\leq \mathbb{\hat{E}}_t[\xi],\ \ \ \ P\text{-a.s.},\ \text{for each}\ \xi\in L_C^1(\Omega_d).
	\end{equation}
\end{lemma}

Next we give the semi-martingale decomposition for generalized  $G$-Brownian motion, which is crucial for our main results.
The proof will be given in Appendix.
\begin{theorem}\label{generalized G-BM decomposition}
	For any $P\in \mathcal{P}$, $B_t$ is a $d$-dimensional semimartingale with  decomposition $B_t=M^P_t+A^P_t$ such that, $P$-a.s., $A^P_t,\langle M\rangle^P_t$ is absolutely continuous with respect to $t$ and
	$$
\frac{dA^P_t}{dt}\in\Sigma \ \ \text{and}\ \ \frac{d\langle M\rangle^P_t}{dt}\in \Gamma\Gamma^T:=\{\gamma\gamma^T:\gamma\in\Gamma \},\ \ \ \ \ {a.e.\  t},\  \text{P-a.s.}
$$
	Here  we denote the quadratic variation of a martingale under $P$ by $\langle \cdot\rangle^P$.
\end{theorem}

Using the above theorem, we can now define the stochastic integral with respect to generalized $G$-Brownian motion.
For each $p\geq1$ and $0<T<\infty$, set
\begin{align*}
M_G^{p,0}(0,T):=& \{\eta:=\eta_t(\omega)=\xi_0I_{\{0\}}+\sum_{j=0}^{N-1}\xi_j(\omega)I_{(t_j,t_{j+1}]}(t),\ \text{for any}\ N\in\mathbb{N},  \\
&  0=t_0\leq t_1\leq \cdots\leq t_N\leq T, \xi_j\in L_{C}^p(\Omega_d)\cap \mathcal{F}_{t_j},j=0,1\cdots,N\}.\\
M_b^{p,0}(0,T):=& \{\eta:=\eta_t(\omega)=\xi_0I_{\{0\}}+\sum_{j=0}^{N-1}\xi_j(\omega)I_{(t_j,t_{j+1}]}(t),\ \text{for any}\ N\in\mathbb{N},  \\
&  0=t_0\leq t_1\leq \cdots\leq t_N\leq T, \xi_j\in L_{b}^p(\Omega_d)\cap \mathcal{F}_{t_j},j=0,1\cdots,N\}.
\end{align*}
For each $\eta\in M_b^{p,0}(0,T)$, set $\|\eta\|_{M_b^{p}(0,T)}:=(\hat
{\mathbb{E}}[\int_{0}^T|\eta_{t}|^{p}dt])^{\frac{1}{p}} $ and denote by
$M_C^{p}(0,T)$ ($M_b^{p}(0,T)$, resp.) the completion of $M_C^{p,0}(0,T)$ ($M_b^{p,0}(0,T)$, resp.) under the norm $\|\cdot
\|_{M_b^{p}(0,T)}$.

Then for each $\eta\in M_b^{2,0}(0,T;\mathbb{R}^d)$, we define the stochastic integral with respect to
$B_t$  as
\[
\int_{0}^{T} \langle \eta_{t},dB_t\rangle:=\sum_{j=0}^{N-1}\langle\xi_{{j}} ,B_{t_{j+1}}-B_{t_{j}
}\rangle,
\]
 which is a linear mapping from $M_b^{2,0}(0,T;\mathbb{R}^d)$ to
${L}_b^{2}(\Omega_d)$. Moreover, it holds that
\begin{proposition}\label{Bcontrol}For each $\eta\in M_b^{2,0}(0,T;\mathbb{R}^d)$ and $Y\in\mathcal{L}(\Omega_d)$, we have
	\begin{equation}
	\hat{\mathbb{E}}[-\beta\int_{0}^{T} |\eta_{t}|d{t}+Y]\leq \hat{\mathbb{E}}[\int_{0}^{T} \langle \eta_{t},dB_t\rangle+Y]\leq \hat{\mathbb{E}}[\beta\int_{0}^{T} |\eta_{t}|d{t}+Y],\end{equation}
	\begin{equation}\label{87676535467}
	\hat{\mathbb{E}}[(\int_{0}^{T} \langle\eta_{t},d B_t\rangle)^2]  \leq 2(\overline{\sigma}^2+\beta^2T)\hat{\mathbb{E}}[\int_0^T|\eta_t|^2dt].
	\end{equation}
\end{proposition}
\begin{proof}	
	We just prove (2), since the proof for (1) is similar. For any $P\in \mathcal{P}$, by Theorem \ref{generalized G-BM decomposition},  we have
	\begin{equation}\label{5678934224}
	\begin{split}
	E_P[(\int_{0}^{T} \langle\eta_{t},dB_t\rangle)^2]&\leq 2E_P[(\int_{0}^{T} \langle\eta_{t},dM_t^P\rangle)^2]+2E_P[(\int_{0}^{T}\langle\eta_{t},dA^P_t\rangle)^2]\\
	&\leq  2E_P[\int_{0}^{T} \langle\eta_{t}\eta_{t}^T,d\langle M^P\rangle^P_t\rangle]+2E_P[(\int_{0}^{T}|\langle\eta_{t},dA^P_t\rangle|)^2]\\
	&\leq  2\overline{\sigma}^2E_P[\int_{0}^{T} |\eta_{t}|^2dt]+2\beta^2T E_P[\int_{0}^{T} |\eta_{t}|^2dt].\end{split}
	\end{equation}
	Taking supremum on both sides, we get the desired result.
\end{proof}

It is worth mentioning that $\int_0^T \langle\eta_t ,d B_t\rangle$ is defined q.s., and under each $P$, it is equivalent to the classical stochastic integral with respect to semi-martingale $B_t$.
By above proposition, the stochastic integral can be extended continuously to $M_b^{2}(0,T;\mathbb{R}^d)$ just as in It\^{o}'s way. Remark that when $\eta\in M_C^2(0,T;\mathbb{R}^d)$, we have $\int_0^T \langle\eta_t ,d B_t\rangle\in L_C^2(\Omega_d)$.
Moreover, we can also obtain stochastic integral on  optional  time interval as \cite{LP}.

The mapping $\tau: \Omega_d\rightarrow [0, \infty)$ is called a stopping time if $(\tau\leq t)\in\mathcal{F}_t$
 and an optional time if $(\tau< t)\in\mathcal{F}_t$ for each $t \geq 0$.
\begin{lemma}\label{stop integral lemma of LP}
For each optional time $\tau$  and $\eta\in M_b^2(0,T;\mathbb{R}^d)$, we have
$
\eta I_{[0,\tau]}\in M_b^2(0,T;\mathbb{R}^d)
$
and
\begin{equation}\label{221111122}
\int_{0}^{\tau\wedge t}\langle \eta_s,dB_s\rangle=\int_{0}^{t}\langle I_{[0,\tau]}\eta_s,dB_s\rangle.
\end{equation}
\end{lemma}
\begin{proof}
The proof is similar to the one of \cite{LP}.
\end{proof}
\begin{remark}
\upshape	{Note that the optional times may do not possess enough continuity  in $\omega$, so in general we cannot expect $
	\eta I_{[0,\tau]}\in M_C^2(0,T;\mathbb{R}^d)$ for $\eta\in M_C^2(0,T;\mathbb{R}^d)$,  see \cite{HWZ}.}
\end{remark}

The quadratic variation process of $B$ is defined by
\begin{equation}
\langle B\rangle_t:=B_tB_t^T-\int_0^tB_sdB^T_s-\int_0^tdB_sB^T_s.
\end{equation}
For any $P\in\mathcal{P}$, we have
$$
\langle B\rangle_t=\langle B\rangle^P_t=\langle  M^P\rangle^P_t,\ \ \ \ P\text{-a.s.}
$$
Hence,
$$\frac{d\langle B\rangle_t}{dt}\in \Gamma\Gamma^T, \ \ \ \ \text{q.s.}
$$
Thus we can define the stochastic integral $\int_0^T\langle \eta_t, d\langle B\rangle_t\rangle $ for $\eta\in M_b^1(0,T;\mathbb{S}(d))$ similarly as the one for $dB_t$ and following property hold:
$$\hat{\mathbb{E}}[|\int_0^T\langle \eta_t, d\langle B\rangle_t\rangle|]\leq \overline{\sigma}^2\sqrt{d} \hat{\mathbb{E}}[\int_0^T|\eta_t| dt].$$

\begin{definition}
	A process $(M_{t})_{t\geq0}$ is called a
	{$G$-martingale} if for each
	$t\in \lbrack0,\infty)$, $M_{t}\in L_{C}^{1}(\Omega_d)\cap \mathcal{F}_{t}$ and for
	each $s\in \lbrack0,t]$, we have
$
	\hat{\mathbb{E}}_s[M_{t}]=M_{s}.
$
\end{definition}

\begin{lemma}\label{G martingale lemma of Peng}
For each $A\in \mathbb{S}(d),p \in \mathbb{R}^d$ and $t\geq 0$,
	\begin{equation}
	\mathbb{\hat{E}}[\frac12\langle A,\langle B\rangle_t\rangle+\langle p,B_t\rangle]=G(A,p)t.
	\end{equation}
\end{lemma}
\begin{proof}By a direct calculation, we have
	\begin{equation}\label{3333333}
	\mathbb{\hat{E}}[\frac12\langle A,\langle B\rangle_t\rangle+\langle p,B_t\rangle]=\sup_{(\gamma,\mu) \in \mathcal{A}^{\Theta}}E_{P^0}[\frac12\int_0^t\langle A,\gamma_s\gamma_s^T\rangle ds+\int_0^t\langle p,\mu_s\rangle ds]\leq G(A,p)t
	\end{equation}
	by the definition of $G$ and the fact that
	$(\gamma_s,\mu_s)$ take values in $\Theta$.
On the other hand, choosing $(\gamma_1,\mu_1)\in \Theta$ so that $G(A,p)=\frac{1}{2}\langle A,\gamma_1{\gamma_1}^T\rangle+\langle p,\mu_1\rangle$
and taking $\gamma_s=\gamma_1s,\mu_s=\mu_1s$, we could get equality in  equation (\ref{3333333}) and this completes the proof.
\end{proof}
\begin{proposition}\label{martingale proposition}
	Let $\eta\in  M_{C}
	^{1}(0,T;\mathbb{S}(d))$, $\zeta\in  M_{C}
	^{2}(0,T;\mathbb{R}^d)$. Then
	\[
	M_{t}:=\int_{0}^{t}\langle \eta_s,d\langle B\rangle_{s}\rangle +2\int_{0}^{t}\langle \zeta_s,d B_{s}\rangle-2\int_{0}^{t}G(\eta_s,\zeta_s)ds
	\]
	is a  $G$-martingale on $[0,T]$.
\end{proposition}
\begin{proof}
By a standard approximation argument, the proof follows from Lemma \ref{G martingale lemma of Peng} and the properties of $\mathbb{\hat{E}}_t$.
\end{proof}

Now we are ready to state the SDEs driven by the generalized $G$-Brownian motion:
\begin{equation}
\label{SDE}
dX^{x}_{t}=b(X_{t}^{x})ds+\sum_{i,j=1}^dh_{ij}(X_{t}^{x})d\langle B^i,B^j\rangle_t+\sum_{j=1}^{d}\sigma_j(X_{t}^{x})dB^j_t,\  X_{0}^{x}=x; \ \ \ t\geq 0,
\end{equation}
where $x\in \mathbb{R}^n$ and
$b,h_{ij}=h_{ji},\sigma_j:\mathbb{R}^n\rightarrow \mathbb{R}^n$ are given Lipschitz functions. Denote by $\sigma=[\sigma_1,\cdots,\sigma_d]$.
From Proposition \ref{Bcontrol} and the contraction mapping method as in \cite{P7}, we can obtain that the $G$-SDE \eqref{SDE} has a unique solution $X\in M^2_C(0,T).$ Moreover, we have the following It\^{o}'s formula for $G$-SDEs
 \eqref{SDE}.
\begin{theorem}\label{Ito formula}Let $f$ be in $C^2(\mathbb{R}^n)$ such that all the second order partial derivatives  satisfy the polynomial growth condition. Then
	\begin{align}
	f(X^x_t)-f(x)=&\int_{0}^t\langle Df(X^x_s),b(X^x_s)\rangle ds+\int^t_0\langle Df(X^x_s),\sigma(X_{s}^{x}) dB_s\rangle
+\sum_{i,j=1}^d\int_{0}^t\langle Df(X^x_s),h_{ij}(X^x_s)d\langle B^i,B^j\rangle_s\rangle\nonumber
\\
& +\frac12\int_{0}^t\langle \sigma(X^x_s)^TD^2f(X^x_s)\sigma(X^x_s),d\langle B\rangle_s\rangle.
	\end{align}

\end{theorem}

Finally, we shall investigate the Markov property for the $G$-SDEs (\ref{SDE}).
Let $\tau$ be an optional time satisfying:
\begin{description}
\item [(H)] $c(\{\tau>T\})\rightarrow 0$, as $ T\rightarrow \infty$.
\end{description}
For each $p\geq1$, we set
$${L}_{C}^{0,p,\tau+}(\Omega_d)=\{X=\sum_{i=1}^n\xi_iI_{A_i}: \ n\in\mathbb{N},\ \{A_i\}_{i=1}^n\text{ is an}\ \mathcal{F}_{\tau+}\text{-partition of}\ \Omega_d,\ \xi_i\in L_C^p(\Omega_d),\ i=1,\cdots,n\}$$
and denote by ${L}_{C}^{p,\tau+}(\Omega)$  the completion of ${L}_{C}^{0,p,\tau+}(\Omega)$ under the norm $||\cdot||_p$.
We also define
$$L^{1,\tau+,*}_C(\Omega):=\{X:\text{there exists}\ X_n\in L^{1,\tau+}_C(\Omega_d)\ \text{such that}\ X_n\uparrow X \ q.s.\}.$$
 By a similar analysis as in \cite{HJL}, the conditional expectation $\hat{\mathbb{E}}_{\tau+}$ is well defined on
 $L^{1,\tau+,*}_C(\Omega_d)$ and can preserve most properties of linear conditional expectation  except the linearity.
 The conditional expectation  $\hat{\mathbb{E}}_{\tau}$ for a stopping time $\tau$  satisfying (H) is defined similarly on $L^{1,\tau,*}_C(\Omega_d)$, where $L^{1,\tau,*}_C(\Omega_d)$ is defined analogously to $L^{1,\tau+,*}_C(\Omega)$ with $\mathcal{F}_{\tau+}$ replaced by $\mathcal{F}_{\tau}$.   Then we have
\begin{theorem}\label{extended BM strongmarkov1}
	Let $Y$ be lower semi-continuous on $\Omega_n$ and bounded from below. Then $Y(X^x_{\tau+\cdot}) \in L^{1,\tau+,*}_C(\Omega_d)$ and
	\begin{equation}
	\hat{\mathbb{E}}_{\tau+}[ Y(X^x_{\tau+\cdot})]=\hat{\mathbb{E}}[ Y(X^y_{\cdot})]_{y=X^x_{\tau}}.
	\end{equation}
	Moreover, if $Y\in C_b(\Omega_n)$, then
	\begin{equation}\label{a3 belong lemma}
	Y(X^x_{\tau+\cdot})\in L^{1,\tau+}_C(\Omega_d),\ \ \ \ \text{and} \ \ \ \ Y(X^x_{\tau+\cdot})\in L^{1,\tau}_C(\Omega_d) \ \text{if furthermore $\tau$ is a stopping time.}
	\end{equation}
\end{theorem}
\begin{proof}
The proof is similar to Lemma 4.1, Theorem 4.2, 4.4   and Corollary 4.8 in \cite{HJL} line by line.
\end{proof}

\section{Exit times for $G$-SDEs}
In this section, we shall give a detailed study of the exit times for $G$-SDEs (\ref{SDE}) from a bounded open set.
For symbol simplicity, we only consider the case where $h_{ij}\equiv 0$ and the results  still hold for the general case.

From now on we always assume that $G$ satisfies the uniformly elliptic condition, i.e.,
there exists three constants $0<\underline{\sigma}^2\leq \overline{\sigma}^2<\infty$ and $\beta\geq 0$ such that, for each $ A_1\geq A_2\in\mathbb{S}(d)$ and $p_1,p_2\in \mathbb{R}^d,$
\begin{equation}\label{uniformly elliptic condition}
\frac{\underline{\sigma}^2}{2} tr(A_1-A_2)-\beta|p_1-p_2|\leq G(A_1,p_1)-G(A_2,p_2)\leq\frac{\overline{\sigma}^2}{2} tr(A_1-A_2)+\beta|p_1-p_2|,
\end{equation}
In fact we can depict the uniform ellipticity of $G$ by $\Theta$: $G$ is uniformly elliptic with parameters $(\underline{\sigma}^2,\overline{\sigma}^2,\beta)$ iff $$\underline{\sigma}^2I_{d\times d}\leq \gamma\gamma^T\leq\overline{\sigma}^2I_{d\times d}\ \text{and} \ |\mu|\leq \beta\ \text{for each}\ (\gamma,\mu)\in\Theta.$$
Then it holds that $$g(p)\leq \beta|p|,\  \frac{\underline{\sigma}^2}{2} tr(A)\leq G_1(A)  \leq\frac{\overline{\sigma}^2}{2} tr(A)\ \text{for}\ A\geq 0\ \text{and}\ \underline{\sigma}^2I_{d\times d}\leq \frac{d\langle B\rangle_t}{dt}\leq\overline{\sigma}^2I_{d\times d},\ \text{q.s.}$$
In the following, we  also assume that $Q$ is a bounded open set in $\mathbb{R}^n$ and $\sigma$ is non-degenerate, i.e., there exists a constant $\lambda>0$ such that
\[
 \lambda I_{n\times n}\leq \sigma(y)\sigma(y)^T, \ \text{for all} \ y\in \overline{Q}.
\]
We will always use $C_f$ ($L_f$, resp.) to denote the bound (the Lipschitz constant, resp.) of a  function $f$ on $\overline{Q}$.
Then we get  that
\begin{equation}\label{myq2}
\lambda I_{n\times n}\leq \sigma(y)\sigma(y)^T\leq C_\sigma^2 I_{n\times n} \ \ \ \ \text{for all} \ y\in \overline{Q}.
\end{equation}

For  each set $D\subset \mathbb{R}^n$ and for any $x\in\mathbb{R}^n$, we define the exit times of $X^x$ by
$$
{\tau}_D^x(\omega):=\inf\{t\geq 0:X^x_t(\omega)\in D^c\},\ \text{for}\ \omega\in \Omega_d.
$$
Now we shall study the properties of ${\tau}_{\overline{Q}}^x$ and ${\tau}_Q^x$.
\begin{lemma}\label{stopping time lemma}
	There exists a constant $C>0$ depending only on $\underline{\sigma}^2,\lambda,\beta,C_b,C_\sigma$ and the diameter of ${Q}$
such that for all $x\in \overline{Q}$,
	\begin{equation}
	\hat{\mathbb{E}}[{{\tau}_{\overline{Q}}^x}]\leq C.
	\end{equation}
\end{lemma}
\begin{proof}
	Without loss of generality, we can assume $0\in Q$. Let $h(y):=Ae^{\alpha y_1}$ on $\overline{Q}$  and take $A,\alpha\geq 0$ large enough such that $\frac{h}{2}(\underline{\sigma}^2\lambda\alpha^2-2\alpha C_b-2\beta\alpha C_\sigma)\geq 1$ on $\overline{Q}$. By  It\^{o}'s formula (extending $h$  to $\mathbb{R}^n$ smoothly if necessary), we have
	\begin{align*}
	h(X^x_{{{\tau}_{\overline{Q}}^x}\wedge t})-h(x)=&\int_0^{{{\tau}_{\overline{Q}}^x}\wedge t}\alpha h( X^{x}_s)\langle \sigma^1(X^x_s),dB_s\rangle+\int_0^{{{\tau}_{\overline{Q}}^x}\wedge t}\alpha h( X^{x}_s)b^1(X^x_s)ds\\
	&+\frac12\int_0^{{{\tau}_{\overline{Q}}^x}\wedge t}\alpha^2h( X^{x}_s)\langle\sigma^1(X^x_s)^T\sigma^1(X^x_s),d\langle B\rangle_s\rangle\\
	\geq & \int_0^{{{\tau}_{\overline{Q}}^x}\wedge t}\alpha h(X^x_s)\langle\sigma^1( X^{x}_s),dB_s\rangle-\int_0^{{{\tau}_{\overline{Q}}^x}\wedge t}\alpha C_bh( X^{x}_s)ds+
\frac12\int_0^{{{\tau}_{\overline{Q}}^x}\wedge t}\alpha^2\underline{\sigma}^2\lambda h( X^{x}_s)ds,
	\end{align*}
	where $\sigma^1,b^1$ is the first row of $\sigma,b$, respectively, and we have used the matrix inequality $\langle A,D\rangle\geq \langle B,D\rangle $ for $A\geq B, D\geq 0$ in the last inequality.
	With the help of Proposition \ref{Bcontrol}, taking expectation to both sides gives that
	\begin{align*}
	2C_h\geq  \hat{\mathbb{E}}[\int_0^{{{\tau}_{\overline{Q}}^x}\wedge t}\frac{h( X^{x}_s)}{2}(\underline{\sigma}^2\lambda\alpha^2-2\alpha C_b-2\beta\alpha C_\sigma)ds]
	\geq \hat{\mathbb{E}}[{{\tau}_{\overline{Q}}^x}\wedge t].
	\end{align*}
	Letting $t\rightarrow\infty$, we get the desired result.
\end{proof}

\begin{lemma}\label{square stopping time lemma}
	There exists a constant $C>0$ depending only on $\underline{\sigma}^2,\lambda,\beta,C_b,C_\sigma$ and  the diameter of ${Q}$  such that for all $x\in \overline{Q}$,
	\begin{equation}
	\hat{\mathbb{E}}[({{\tau}_{\overline{Q}}^x})^2]\leq C.
	\end{equation}
\end{lemma}
\begin{proof}
	Without loss of generality, we can assume $0\in Q$. Consider $th(y)$, where $h$ with  $A,\alpha$  is  assumed as in the proof of Lemma \ref{stopping time lemma}. By It\^{o}'s formula, we have
	\begin{align*}
	({{\tau}_{\overline{Q}}^x}\wedge t)h(X^x_{{{\tau}_{\overline{Q}}^x}\wedge t})  =&\int_0^{{{\tau}_{\overline{Q}}^x}\wedge t}h( X^{x}_s)ds+\int_0^{{{\tau}_{\overline{Q}}^x}\wedge t}s\alpha h( X^{x}_s)\langle\sigma^1(X^x_s),dB_s\rangle+\int_0^{{{\tau}_{\overline{Q}}^x}\wedge t}s\alpha h( X^{x}_s)b^1(X^x_s)ds\\
	&+\frac12\int_0^{{{\tau}_{\overline{Q}}^x}\wedge t}s\alpha^2h( X^{x}_s)\langle\sigma^1(X^x_s)^T\sigma^1(X^x_s),d\langle B\rangle_s\rangle\\
	\geq &\int_0^{{{\tau}_{\overline{Q}}^x}\wedge t}s\alpha h( X^{x}_s)\langle\sigma^1(X^x_s),dB_s\rangle-\int_0^{{{\tau}_{\overline{Q}}^x}\wedge t}s\alpha C_b h( X^{x}_s)ds+\frac12\int_0^{{{\tau}_{\overline{Q}}^x}\wedge t}s\alpha^2 \underline{\sigma}^2\lambda h( X^{x}_s)ds.
	\end{align*}
	Taking expectation on both sides, we get that
	\begin{align*}
	C_h\hat{\mathbb{E}}[{{\tau}_{\overline{Q}}^x}\wedge t]\geq \hat{\mathbb{E}}[({{\tau}_{\overline{Q}}^x}\wedge t)h(X^x_{{{\tau}_{\overline{Q}}^x}\wedge t})]\geq \hat{\mathbb{E}}[\int_0^{{{\tau}_{\overline{Q}}^x}\wedge t}sds]=\frac{1}{2}\hat{\mathbb{E}}[({{{\tau}_{\overline{Q}}^x}\wedge t})^2].
	\end{align*}
	Letting $t\rightarrow \infty$  and we obtain that
	$$
	\hat{\mathbb{E}}[({{{\tau}_{\overline{Q}}^x}})^2]\leq 2C_h\hat{\mathbb{E}}[{{\tau}_{\overline{Q}}^x}],
	$$
which together with Lemma \ref{stopping time lemma} imply the desired result.
\end{proof}
\begin{remark}\upshape{By Theorem \ref{LG characteriazation theorem}, we know that
${\tau}^{x}_{\overline{Q}} \in L_b^1(\Omega_d)$ for any $x\in\mathbb{R}^n$, since the case that  $x\in \overline{Q}^c$  is trivial.
}
\end{remark}

In order to state the main result, we need the following additional condition on the bounded open set $Q$.
 \begin{itemize}
\item[$\cdot$] An open set $O$ is said to satisfy the exterior ball condition if for all $x\in \partial O$,  there exists an open ball $U(z,r)$ such that
$U(z,r)\subset O^c$ and $x\in \partial U(z,r)$.
\end{itemize}
In the rest of this paper, we always assume that $Q$ satisfies the exterior ball condition.
The following tells us that $G$-SDEs originating  at the boundary point with exterior ball will exit $\overline{Q}$ immediately.
\begin{lemma}\label{stopping time lemma 2}
For each $x\in \partial Q$, we have q.s. $\tau^x_{\overline{Q}}=0$, i.e.,
for each $\varepsilon>0$, there exists a point $t\in(0,\varepsilon]$ such that $X^x_{t}\in \overline{Q}^c$.
\end{lemma}
\begin{proof}
	Assume $U(z,r)$  is the exterior ball of $Q$ at $x$. We are going to prove the conclusion by a technique from  Lions and   Menaldi \cite{LM}.
We set $h(y):=e^{-k|y-z|^2}$, where the constant $k$ will be  determined in the sequel.
Then we have
\begin{align*}
&D_yh(y)=-2k(y-z)e^{-k|y-z|^2},\\
&D^2_{yy}h(y)=(4k^2(y_i-z_i)(y_j-z_j)-2k\delta_{ij})e^{-k|y-z|^2}
=(4k^2(y_i-z_i)(y_j-z_j))e^{-k|y-z|^2}-(2k\delta_{ij})e^{-k|y-z|^2}.
\end{align*}

Note that the matrix $(4k^2(y_i-z_i)(y_j-z_j))e^{-k|y-z|^2}=4k^2(y-z)(y-z)^Te^{-k|y-z|^2}$ and $(2k\delta_{ij})e^{-k|y-z|^2}$ are nonnegative.
Choosing $k$ large enough, we can find some  constant $\mu>0$ so that for all $y\in \overline{Q}$,
\begin{align*}
&\langle\sigma(y)^TD^2_{yy}h(y)\sigma(y),\frac{d\langle B\rangle_t}{dt}\rangle+2\langle D_{y}h(y),b(y)\rangle-2\beta|\sigma^T(y)||D_{y}h(y)| \\
& =\langle\sigma(y)^T(4k^2(y_i-z_i)(y_j-z_j))\sigma(y),\frac{d\langle B\rangle_t}{dt}\rangle e^{-k|y-z|^2}-\langle\sigma(y)^T(2k\delta_{ij})\sigma(y),\frac{d\langle B\rangle_t}{dt}\rangle e^{-k|y-z|^2}\\
&\ \ \ \ -4k\langle(y-z),b(y) \rangle e^{-k|y-z|^2}-2\beta|\sigma^T(y)||D_{y}h(y)|\\
& \geq (4\underline{\sigma}^2\lambda k^2|y-z|^2-4k(C_{b}+\beta C_{\sigma} )|y-z|-2k\overline{\sigma}^2C^2_{\sigma} )e^{-k|y-z|^2}\geq \mu.
\end{align*}
Then applying It\^{o}'s formula, we derive that
\begin{align*}
h(X^x_{\tau^x_{\overline{Q}}\wedge t})-h(x)
&=\int_{0}^{\tau^x_{\overline{Q}}\wedge t} \langle D_yh(X^x_s),\sigma(X^x_s)dB_s\rangle+\int_{0}^{\tau^x_{\overline{Q}}\wedge t} \langle D_yh(X^x_s),b(X^x_s)\rangle ds\\
&\ \ \ +\frac12\int_{0}^{\tau^x_{\overline{Q}}\wedge t}\langle\sigma(X^x_s)^TD_{yy}^2h(X^x_s)\sigma(X^x_s),d\langle B\rangle_s\rangle.
\end{align*}
Taking expectation to both sides and using Proposition \ref{Bcontrol}, we conclude that
$$\frac{\mu}{2}\hat{\mathbb{E}}[{\tau^x_{\overline{Q}}\wedge t}]\leq \hat{\mathbb{E}}[h(X_{\tau^x_{\overline{Q}}\wedge t})-h(x)]\leq 0,$$
since $h(y)-h(x)\leq 0$ for   $y\in U(z,r)^c$.
Therefore, it holds that
$$\hat{\mathbb{E}}[{\tau^x_{\overline{Q}}}\wedge t]\leq 0.$$ Letting $t\rightarrow \infty$, we obtain
$\hat{\mathbb{E}}[\tau^x_{\overline{Q}}]\leq 0$, from which we get that $\tau^x_{\overline{Q}}=0$ q.s.
The proof is complete.
\end{proof}

Lemma \ref{stopping time lemma 2} indicates that ${\tau}_Q^x= {\tau}_{\overline{Q}}^x$ for the boundary points of $Q$.
In the following,  we shall show that it also remains true for inner points of $Q$.
\begin{theorem}\label{exit times equal lemma}For each $x\in \mathbb{R}^n$, we have
	$${\tau}_Q^x= {\tau}_{\overline{Q}}^x, \ \text{q.s.}$$
\end{theorem}

In order to prove Theorem \ref{exit times equal lemma}, we will study the continuity of ${\tau}_Q^x$ in $\omega$. For this purpose,
 we shall consider the image space $\Omega_n$ of $G$-SDE (\ref{SDE}). Denote by $B'$ the canonical process on $\Omega_n$. For each subset $D$ of $\mathbb{R}^n$, we define on $\Omega_n$ the exit  times of $B'$ by
$$
{\tau}_D^{x,1}(\omega):=\inf\{t\geq 0:x+{\omega_t}\in D^c\},\ \text{for}\ \omega\in \Omega_n.
$$
Then we have that ${\tau}_D^{x}(\omega)={\tau}_D^{0,1}(X^x_{\cdot}(\omega))$. We need following lemmas to complete the proof of Theorem \ref{exit times equal lemma}.

\begin{lemma}\label{semi-continuity of exit time}
On $\Omega_n$, ${\tau}_Q^{x,1}$ is lower semi-continuous and ${\tau}_{\overline{Q}}^{x,1}$ is upper semi-continuous.
\end{lemma}
\begin{proof}
We just prove that ${\tau}_{\overline{Q}}^{x,1}$ is upper semi-continuous, since the proof of another part is similar.

For any given $\omega\in\Omega_n$, set $t_0:={\tau}_{\overline{Q}}^{x,2}(\omega)$. It suffices to consider the case where $t_0<\infty$.
Then we can find an arbitrarily small $\varepsilon>0$ such
that  $x+B'_{t_0+\varepsilon}(\omega)\in\overline{Q}^c$. Since $\overline{Q}^c$ is open, there exists an open ball $U(x+B'_{t_0+\varepsilon}(\omega),r)$ with center $x+B'_{t_0+\varepsilon}(\omega)$ and radius  $r$ such that $U(x+B'_{t_0+\varepsilon}(\omega),r)\subset \overline{Q}^c$. For each ${\omega}'$ whose distance with $\omega$ is small  enough,  we will have $x+B'_{t_0+\varepsilon}({\omega}')\in U(x+B'_{t_0+\varepsilon}(\omega),r)\subset\overline{Q}^c$.
That is, ${\tau}_{\overline{Q}}^{x,1}({\omega}')\leq t_0+\varepsilon$. This completes the proof.
\end{proof}

By Kolmogorov's criterion, $(X^x_t)_{t\geq 0}$ induces a tight family of probabilities $\mathcal{P}\circ(X_{\cdot}^x)^{-1}:=\{P\circ (X_{\cdot}^x)^{-1}: P\in\mathcal{P}\}$ on $\Omega_n$. We denote the induced upper capacity  by $c^x_2:=c_{\mathcal{P}\circ(X_{\cdot}^x)^{-1}}=\sup_{P\in\mathcal{P}} P\circ (X_{\cdot}^x)^{-1} $ and the induced upper expectation by $\hat{\mathbb{E}}^x_2:=\hat{\mathbb{E}}_{\mathcal{P}\circ(X_{\cdot}^x)^{-1}}=\sup_{P\in\mathcal{P}}E_{P\circ (X_{\cdot}^x)^{-1}}$. More generally, for a set $A\in \mathbb{R}^n$, we define
$\mathcal{P}^A_2:=\cup_{x\in A}\mathcal{P}\circ(X_{\cdot}^x)^{-1}$, and $
\hat{\mathbb{E}}^A_2:=\hat{\mathbb{E}}_{\mathcal{P}^A_2}=\sup_{P\in\mathcal{P}^A_2}E_{P}$ as well as  $
c^A_2:=c_{\mathcal{P}^A_2}=\sup_{P\in\mathcal{P}^A_2}{P}$.

\begin{lemma}\label{SDE continuity wrt y}
Assume $(y_k)_{k\geq 1}$ is a sequence in $\mathbb{R}^n$  such that $|y_k-y|\rightarrow 0$ for some $y$. Then for each $\varphi\in C_b(\Omega_n)$, we have
$$\mathbb{\hat{E}}[|\varphi(X^y_{\cdot})-\varphi(X^{y_k}_{\cdot})|]\rightarrow 0.$$
\end{lemma}
\begin{proof}
By Lemma 3.1 in Chap VI of \cite{P7}, we can choose a sequence of  $\varphi_m\in C_b(\Omega_n)$  such that $|\varphi_{m}|\leq C_\varphi$, $0\leq |\varphi_{m}(\omega)-\varphi_{m}({\omega}')|\leq m||\omega-{\omega}'||_{C[0,m]}$ and $\varphi_{m}\uparrow \varphi$, as $m\rightarrow \infty$. We  pick a compact set $K\subset \mathbb{R}^n$ such that $y_k,y\in K$ for each $k\geq 1$ and then the family
	$
	\mathcal{P}^{K}_2
	$
is tight by Kolmogorov's criterion.
	Then for any fixed $\varepsilon>0$, there is a compact set $\widetilde{K}\subset\Omega_n$ such that  $c^{K}_2(\widetilde{K}^c)\leq\varepsilon$, which implies $c^z_2(\widetilde{K}^c)\leq\varepsilon$ uniformly for $z\in K$. By Dini's theorem,
	$\varphi_m\uparrow \varphi$ uniformly on $\widetilde{K}$. So we can take $m$ large enough such that $0\leq \varphi-\varphi_m\leq \varepsilon$ on $\widetilde{K}$. Then by the basic estimate for $G$-SDEs, we obtain some constant $C\geq 0$ such that
	\begin{align*}
	\mathbb{\hat{E}}[|\varphi(X^y_{\cdot})-\varphi(X^{y_k}_{\cdot})|]
	&\leq \mathbb{\hat{E}}^y_2[|\varphi-\varphi_m|]+\mathbb{\hat{E}}[|\varphi_m(X^y_\cdot)-\varphi_m(X^{y_k}_\cdot)|]+\mathbb{\hat{E}}^{y_k}_2[|\varphi-\varphi_m|]\\
	&\leq \mathbb{\hat{E}}^y_2[|\varphi-\varphi_m|I_{\widetilde{K}}]+mC|y-y_k|+\mathbb{\hat{E}}^{y_k}_2[|\varphi-\varphi_m|I_{\widetilde{K}}]+2C_\varphi c^{y}_2({\widetilde{K}}^c)+2C_\varphi c^{y_k}_2({\widetilde{K}}^c)\\
	&\leq 2\varepsilon +mC|y-y_k|+4C_\varphi\varepsilon.
	\end{align*}
	Letting $k\rightarrow \infty$, we obtain that
	$$
	\limsup\limits_{k\rightarrow\infty}\mathbb{\hat{E}}[|\varphi(X^y_{\cdot})-\varphi(X^{y_k}_{\cdot})|]\leq 2\varepsilon +4C_\varphi\varepsilon.
	$$
	Since $\varepsilon$ can be arbitrary small, we obtain the desired lemma.
	\end{proof}
\begin{lemma}\label{SDE belong to G space}
	Assume  $\varphi\in C_b(\Omega_n)$. Then it holds that
	$\varphi(X^x_\cdot)\in L_C^1(\Omega_d).$
\end{lemma}
\begin{proof}
This follows from (\ref{a3 belong lemma}) for stopping time $\tau=0$ in Theorem \ref{extended BM strongmarkov1}.
	\end{proof}

 Now we are in a position to state the proof of Theorem \ref{exit times equal lemma}.

\begin{proof}[The proof of  Theorem \ref{exit times equal lemma}]
The case that $x\in \overline{Q}^c$ is trivial and the case that $x\in \partial{Q}$ is from Lemma \ref{stopping time lemma 2}.
Then we just need to consider the case that $x\in Q$. It suffices to prove  that $\mathbb{\hat{E}}[({\tau}_{\overline{Q}}^x-{\tau}_Q^x)\wedge t]=0$ for each $t>0$.

Denote $\delta_t=({\tau}_{\overline{Q}}^{0,1}-{\tau}_Q^{0,1})\wedge t$, then $({\tau}_{\overline{Q}}^x-{\tau}_Q^x)\wedge t=\delta_t(X^x_\cdot)=\delta_t(X^x_{{\tau}_Q^x+\cdot})$ by the definition. Since $\delta_t$ is bounded and upper semi-continuous on $\Omega_n$, we can find a sequence of continuous functions $(f_m)_{m\geq 1}$ on $\Omega_n$ such that $0\leq f_m \leq 2t$  and $f_m\downarrow \delta_t$. Then it follows from Theorem \ref{extended BM strongmarkov1} that,
\begin{align*}
\mathbb{\hat{E}}[({\tau}_{\overline{Q}}^x-{\tau}_Q^x)\wedge t]=\mathbb{\hat{E}}[\delta_t(X^x_{{\tau}_Q^x+\cdot})]
\leq \mathbb{\hat{E}}[f_m(X^x_{{\tau}_Q^x+\cdot})]=\mathbb{\hat{E}}[\mathbb{\hat{E}}[f_m(X^y_{\cdot})]_{y=X^x_{{\tau}_Q^x}}], \ \ \ \text{for all}\ m\geq 1.
\end{align*}

Denote $\varphi_m(y)=\mathbb{\hat{E}}[f_m(X^y_{\cdot})]$ for   $y\in \mathbb{R}^n$.
Recalling Lemma \ref{SDE belong to G space}, we have $f_m(X^y_{\cdot})\in L^1_C(\Omega_d)$. Then   Proposition \ref{downward convergence proposition} and Lemma \ref{stopping time lemma 2} imply that for each $y\in\partial Q$
\[
\varphi_m(y)\downarrow \mathbb{\hat{E}}[\delta_t(X^y_{\cdot})]=0, \ \text{as}\ m\rightarrow\infty.
\]
 Since $\varphi_m$ is continuous on $\partial Q$ by Lemma \ref{SDE continuity wrt y}, we derive that
$\varphi_m(y)\downarrow 0$ uniformly on $\partial{Q}$ by Dini's theorem. Consequently, we deduce that
$$
\mathbb{\hat{E}}[\mathbb{\hat{E}}[f_m(X^y_{\cdot})]_{y=X^x_{{\tau}_Q^x}}]=\mathbb{\hat{E}}[\varphi_m({X^x_{{\tau}_Q^x}})]\downarrow 0, \ \text{as}\ m\rightarrow\infty,
$$
which implies the desired result.
\end{proof}

Now we are going to show that the exit times are quasi-continuous.
\begin{lemma}\label{induced probability measures weakly compact}
If $K$ is a compact set in $ \mathbb{R}^n$, then  the set $\mathcal{P}^K_2$ is weakly compact on $\Omega_n$.
\end{lemma}
\begin{proof}
Let $(P_k\circ(X_{\cdot}^{x_k})^{-1})_{k\geq 1}$ be any sequence in $\mathcal{P}^K_2$.
Since $K$ is compact, we can find a subsequence $x_{k_m}$ such that $|x_{k_m}-x|\rightarrow 0$ for some $x\in K$.
Note that $\mathcal{P}$ is weakly compact, there is a subsequence $P_{k_{m_l}}\in \mathcal{P}$ such that $P_{k_{m_l}}$ converges to some $P\in \mathcal{P}$.
For any $\varphi\in C_b(\Omega_n)$, note that $\varphi(X_{\cdot}^{x})\in L^1_C(\Omega_d)$. Then in view of Lemma \ref{SDE continuity wrt y} and Lemma 29 in \cite{DHP}, we get that
\begin{align*}
&\lim\limits_{l\rightarrow\infty}|E_{P_{{k_{m_l}}}\circ(X_{\cdot}^{x_{{k_{m_l}}}})^{-1}}[\varphi]-E_{{P}\circ(X_{\cdot}^{x})^{-1}}[\varphi]|\\
&\leq \lim\limits_{l\rightarrow\infty}|E_{P_{{k_{m_l}}}\circ(X_{\cdot}^{x_{{k_{m_l}}}})^{-1}}[\varphi]-E_{P_{{k_{m_l}}}\circ(X_{\cdot}^{x})^{-1}}[\varphi]|
+\lim\limits_{l\rightarrow\infty}|E_{P_{{k_{m_l}}}\circ(X_{\cdot}^{x})^{-1}}[\varphi]-E_{{P}\circ(X_{\cdot}^{x})^{-1}}[\varphi]|\\
&\leq \lim\limits_{l\rightarrow\infty}|E_{P_{{k_{m_l}}}}[\varphi(X_{\cdot}^{x_{{k_{m_l}}}})]-E_{P_{{k_{m_l}}}}[\varphi(X_{\cdot}^{x})]|
+\lim\limits_{l\rightarrow\infty}|E_{P_{{k_{m_l}}}}[\varphi(X_{\cdot}^{x})]-E_{P}[\varphi(X_{\cdot}^{x})]|\\
&\leq \lim\limits_{l\rightarrow\infty}\mathbb{\hat{E}}[|\varphi(X_{\cdot}^{x_{{k_{m_l}}}})-\varphi(X_{\cdot}^{x})|]
=0,
\end{align*}
which ends the proof.
\end{proof}
\begin{theorem}\label{exit time quasi continuity lemma}
Let $K$ be a bounded set in $ \mathbb{R}^n$. Then ${\tau}_Q^{0,1}$ and ${\tau}_{\overline{Q}}^{0,1}$ both belong to $L_C^1(\Omega_n,\mathcal{P}^K_2)$.
\end{theorem}
\begin{proof}
We just need to prove the case that $K$ is a compact set since  a bounded set is contained in some compact set.
Let $\Gamma=\{{\tau}_Q^{0,1}={\tau}_{\overline{Q}}^{0,1}\}$. Then  $c^K_2(\Gamma^c)=\sup_{x\in K}c^x_2(\Gamma^c)=\sup_{x\in K}c(\{{\tau}_Q^{x}<{\tau}_{\overline{Q}}^{x}\})=0$ by Theorem
\ref{exit times equal lemma}. Moreover, we can write the polar set as
$$\Gamma^c=\{{\tau}_Q^{0,1}<{\tau}_{\overline{Q}}^{0,1}\}=\bigcup_{s<r;s,r\in \mathbb{Q}}\{{\tau}_Q^{0,1}\leq s\}\cap\{{\tau}_{\overline{Q}}^{0,1}\geq r\}.$$
By the semi-continuity of ${\tau}_Q^{0,1}$ and ${\tau}_{\overline{Q}}^{0,1}$, we conclude that $\{{\tau}_Q^{0,1}\leq s\}\cap\{{\tau}_{\overline{Q}}^{0,1}\geq r\}$ is closed.
Note that $c^K_2$ is weakly compact by Lemma \ref{induced probability measures weakly compact}. Then according to Proposition \ref{downward convergence proposition},
for any $\varepsilon>0$, there exists an open set $O\supset{\Gamma}^c$ such that $c^K_2(O)<\frac\varepsilon 2$.
By Lemma \ref{stopping time lemma}, we can take $k$ large enough such that
$
c^K_2({\tau}_Q^{0,1}>k)\leq \frac\varepsilon 2.
$
Set $F=O^c\cap \{{\tau}_Q^{0,1}\leq k \}$. It is obvious that $c^K_2(F^c)\leq \varepsilon $ and ${\tau}_Q^{0,1}={\tau}_{\overline{Q}}^{0,1}$ are continuous on $F$.

Recalling Lemma \ref{square stopping time lemma}, we conclude that
$$ \mathbb{\hat{E}}^K_2[{{\tau}_{{Q}}^{0,1}}I_{\{{{\tau}_{{Q}}^{0,1}}>N\}}]\leq\mathbb{\hat{E}}^K_2[{{\tau}_{\overline{Q}}^{0,1}}I_{\{{{\tau}_{\overline{Q}}^{0,1}}>N\}}]\leq \frac{\mathbb{\hat{E}}^K_2[|{{\tau}_{\overline{Q}}^{0,1}}|^2]}{N}=\frac{\sup_{x\in K}\mathbb{\hat{E}}[|{{\tau}_{\overline{Q}}^{x}}|^2]}{N}\rightarrow 0,\ \text{as}\ N\rightarrow \infty,$$
which together with the characterization of $L_C^1(\Omega_n,\mathcal{P}^K_2)$ (Theorem \ref{LG characteriazation theorem}) imply the desired the result.
\end{proof}

Finally we study the continuity property of ${\tau}^{x}_Q$ with respect to $x$. For each $\varepsilon>0$,
 we denote $Q_\varepsilon:=\{x\in Q:dist(x,\partial Q)>\varepsilon\}$ and $Q_{-\varepsilon}:=\{x\in \mathbb{R}^n:dist(x,Q)<\varepsilon\}$.
Then the exterior ball condition can be preserved for the following approximation from inside.
\begin{lemma}
For any $\varepsilon>0$,	$Q_\varepsilon$ also satisfies the exterior ball condition.
\end{lemma}
\begin{proof}
	Let $x$ be in $\partial Q_\varepsilon$.
Then there exists a point $x'\in \partial Q$ such that $d(x,x')=\varepsilon$.
Assume that $U(y,r)$ is the exterior ball of $Q$ at $x'$. We claim that
	$U(y+(x-x'),r)=U(y,r)+(x-x')$ is the   exterior ball of $Q_\varepsilon$ at $x$. Indeed, for any $z\in Q_\varepsilon$, we have $z+(x'-x)\in Q$ and then
	$$
	d(z,y+x-x')=d(z+(x'-x),y)>r.
	$$
The proof is complete.
\end{proof}

For any fixed $T>0$, by a standard argument we can find some constant $C_T$ depending on $T$  such that
$$
\hat{\mathbb{E}}[\sup_{0\leq t\leq T}|X^x_t-X^y_t|^{n+1}]\leq C_T|x-y|^{n+1}.
$$
It follows from Kolmogorov's criterion for continuity  that for any fixed $\alpha\in (0,\frac{1}{n+1})$
\begin{equation}\label{887268439202}
\hat{\mathbb{E}}[\eta_T^{n+1}]<\infty,\ \ \ \text{where}\ \eta_T:=\sup_{x\neq y}\frac{\sup_{0\leq t\leq T}|X^x_t-X^y_t|}{|x-y|^\alpha}.
\end{equation}
From this and Lemma \ref{semi-continuity of exit time}, it is easy to prove that, q.s., ${\tau}^{x}_Q$ and $ {\tau}^{x}_{\overline{Q}}$ are lower and upper-continuous  with respect to $x$.
Then Theorem \ref{exit times equal lemma} implies that, \begin{equation}\label{21312423423}
{\tau}^{x_k}_Q\rightarrow {\tau}^{x}_Q, \ \ \ \  \text{q.s.}
\end{equation}
whenever $|x_k- x|\rightarrow 0$. Moreover, we have
\begin{lemma}\label{tau continuous lemma}
	Assume $|x_k- x|\rightarrow 0$. Then
	\begin{equation}
	\hat{\mathbb{E}}[{{\tau}^{x}_Q\vee {\tau}^{x_k}_Q}-{{\tau}^{x}_Q\wedge {\tau}^{x_k}_Q}]
	\rightarrow 0, \ \text{as}\ k\rightarrow \infty.
	\end{equation}
\end{lemma}
\begin{proof}
	For any $L>0,T>0$ and $\varepsilon>0$, let $\alpha$  and $\eta_T$ be defined as in (\ref{887268439202}). We consider the set that ${Q}_{-L{\varepsilon}^\alpha}=\{v\in\mathbb{R}^n:dist(v,Q)<L{\varepsilon}^\alpha\}$.
 For any $y$ such that $|x-y|\leq \varepsilon$, on $\{\eta_T\leq L\}\cap \{{\tau}^{y}_Q\leq T \}$ we have that
 \[
 \sup_{0\leq t\leq {\tau}^{y}_Q}|X^y_{t}-X^x_{t}|\leq L\varepsilon^{\alpha},
 \]
 which implies that ${\tau}^{y}_Q\leq {\tau}^{x}_{{Q}_{-L{\varepsilon}^\alpha}}$. Similarly, for
	$Q_{L{\varepsilon}^\alpha}=\{v\in Q:dist(v,Q)>L{\varepsilon}^\alpha\}$, we have ${\tau}^{x}_{{Q}_{L{\varepsilon}^\alpha}}\leq{\tau}^{y}_Q$ on  $\{\eta_T\leq L\}\cap \{{\tau}^{y}_Q\leq T \}$.
	
	For each $Q_{-L{\varepsilon}^\alpha}$, take a bounded open set $\widetilde{Q}_{-L{\varepsilon}^\alpha}$ with smooth boundary such that $Q_{-L{\varepsilon}^\alpha}\subset \widetilde{Q}_{-L{\varepsilon}^\alpha}$ and $\widetilde{Q}_{-L{\varepsilon}^\alpha}\downarrow \overline{Q}$. It follows from  Theorem \ref{exit time quasi continuity lemma} that
	${\tau}^{0,1}_{\widetilde{Q}_{-L{\varepsilon}^\alpha}}-{\tau}^{0,1}_{Q_{L{\varepsilon}^\alpha}}\in L_C^1(\Omega_n,\mathcal{P}\circ  (X^x_{\cdot})^{-1})$ since ${\widetilde{Q}_{-L{\varepsilon}^\alpha}}$ and ${Q_{L{\varepsilon}^\alpha}}$ both satisfy the exterior ball condition.
	Then we get that
	\begin{equation}
	\label{4354536576867}
	\begin{split}
	&\hat{\mathbb{E}}[{{\tau}^{x}_Q\vee {\tau}^{x_k}_Q}-{{\tau}^{x}_Q\wedge {\tau}^{x_k}_Q}]\\
	&\leq \hat{\mathbb{E}}[({{\tau}^{x}_Q\vee {\tau}^{x_k}_Q}-{{\tau}^{x}_Q\wedge {\tau}^{x_k}_Q})I_{{\{ {\tau}^{x_k}_{{Q}}\leq  T\}}}I_{\{\eta_T\leq L\}}]+\hat{\mathbb{E}}[({{\tau}^{x}_Q\vee {\tau}^{x_k}_Q}-{{\tau}^{x}_Q\wedge {\tau}^{x_k}_Q})I_{{\{ {\tau}^{x_k}_{{Q}}\leq  T\}}}I_{\{\eta_T> L\}}]\\
	&\ \ \  \ +\hat{\mathbb{E}}[({{\tau}^{x}_Q\vee {\tau}^{x_k}_Q}-{{\tau}^{x}_Q\wedge {\tau}^{x_k}_Q})I_{{\{ {\tau}^{x_k}_{{Q}}> T\}}}]\\
	&\leq \hat{\mathbb{E}}[({{\tau}^{x}_Q\vee {\tau}^{x_k}_Q}-{{\tau}^{x}_Q\wedge {\tau}^{x_k}_Q})I_{{\{ {\tau}^{x_k}_{{Q}}\leq  T\}}}I_{\{\eta_T\leq L\}}]+\hat{\mathbb{E}}[{{\tau}^{x}_Q\vee {\tau}^{x_k}_Q}I_{\{\eta_T> L\}}]+\hat{\mathbb{E}}[{{\tau}^{x}_Q\vee {\tau}^{x_k}_Q}I_{{\{ {\tau}^{x_k}_{{Q}}> T\}}}]\\
	&=:I_1+I_2+I_3.
	\end{split}
	\end{equation}
	For $I_1$, we take $k$ large enough such that $|x_k-x|\leq \varepsilon$. Then for any $T$ and $L$, it follows that
	\begin{align*}
	I_1\leq \hat{\mathbb{E}}[({\tau}^{x}_{\widetilde{Q}_{-L{\varepsilon}^\alpha}}-{\tau}^{x}_{Q_{L{\varepsilon}^\alpha}})I_{{\{ {\tau}^{x_k}_{{Q}}\leq  T\}}}I_{\{\eta_T\leq L\}}]
	\leq \hat{\mathbb{E}}[{\tau}^{x}_{\widetilde{Q}_{-L{\varepsilon}^\alpha}}-{\tau}^{x}_{Q_{L{\varepsilon}^\alpha}}]
	= \hat{\mathbb{E}}^x_2[{\tau}^{0,1}_{\widetilde{Q}_{-L{\varepsilon}^\alpha}}-{\tau}^{0,1}_{Q_{L{\varepsilon}^\alpha}}],
	\end{align*}
which indicates that for each $\varepsilon>0$
\[
\limsup_{k\rightarrow\infty}I_1\leq \hat{\mathbb{E}}^x_2[{\tau}^{0,1}_{\widetilde{Q}_{-L{\varepsilon}^\alpha}}-{\tau}^{0,1}_{Q_{L{\varepsilon}^\alpha}}].
\]
Sending $\varepsilon\rightarrow 0$ and using Proposition \ref{downward convergence proposition} and   Theorem \ref{exit times equal lemma}, we get that
\[
\limsup_{k\rightarrow\infty}I_1\leq \hat{\mathbb{E}}^x_2[{\tau}^{0,1}_{{{\overline{Q}}}}-{\tau}^{0,1}_{Q}]=\hat{\mathbb{E}}[{\tau}^{x}_{{{\overline{Q}}}}-{\tau}^{x}_{Q}]= 0.
\]
	For any $\delta>0$, by Lemma \ref{square stopping time lemma} and Proposition 19 in \cite{DHP}  we can first take $T$ large enough such that $I_3\leq \delta$ and then take $L$ large enough such that $I_2\leq \delta$ for each $k$. Now letting $k\rightarrow \infty$ in  (\ref{4354536576867}), we get that
	$$
	\limsup_{k\rightarrow\infty}\hat{\mathbb{E}}[{{\tau}^{x}_Q\vee {\tau}^{x_k}_Q}-{{\tau}^{x}_Q\wedge {\tau}^{x_k}_Q}]\leq 2\delta.
	$$
	Sending $\delta\rightarrow 0$ and we get the desired result.
\end{proof}

\section{Application to probabilistic representations for PDEs}
This section is devoted to studying the relationship between SDEs driven by generalized $G$-Brownian motion and fully nonlinear elliptic equations.
In fact, with the help of the results of the previous sections, we shall introduce  a stochastic representation for a class of fully nonlinear elliptic equations with Dirichlet boundary.

The following results are important for our future discussion. First, we shall extend the
Theorem \ref{extended BM strongmarkov1} to a more general case.
\begin{theorem}\label{extended strong markov theorem0}
	Let $\tau$ be an optional time satisfying assumption (H).
Then for each  $Y\in  {L}_C^1(\Omega_n,\mathcal{P}\circ (X^x_{\tau+\cdot})^{-1})$,
	\begin{equation}
	Y(X^x_{\tau+\cdot})\in L^{1,\tau+}_C(\Omega_d)\ \ \ \  \text{and}\ \ \ \  \hat{\mathbb{E}}_{\tau+}[Y(X^x_{\tau+\cdot})]=\hat{\mathbb{E}}[ Y(X^y_{\cdot})]_{y=X^x_{\tau}}.
	\end{equation}
	Moreover, if $\tau$ is also a stopping time, we have $	Y(X^x_{\tau+\cdot})\in L^{1,\tau}_C(\Omega_d)$.
\end{theorem}
\begin{proof} The proof shall be divided into the following three steps.

{\it 1 Bounded case.}
Suppose $Y$ is bounded by some constant $C_Y$. Then by Theorem \ref{LG characteriazation theorem}, for any $\varepsilon>0$ we can pick a closed set $D$  in $\Omega_n$ such that $c(O)\leq \varepsilon$
and $Y|_{D}$ is continuous, where $O:=(X^x_{\tau+\cdot})^{-1}(D^c)$.
By Tietze extension theorem, there is a continuous function  $\widetilde{Y}$ on $\Omega_n$  such that $\widetilde{Y}=Y$ on $D$ and $|\widetilde{Y}|\leq C_Y$.
Recalling Theorem \ref{extended BM strongmarkov1}, we get that
	$$\hat{\mathbb{E}}_{\tau+}[\widetilde{Y}(X^x_{\tau+\cdot})]=\hat{\mathbb{E}}[\widetilde{Y} (X^y_{\cdot})]_{y=X^x_{\tau}}.$$
For the left side, it holds that
	\begin{align*}
\hat{\mathbb{E}}[|\widetilde{Y}(X^x_{\tau+\cdot})-Y(X^x_{\tau+\cdot})|]
	\leq \hat{\mathbb{E}}[|\widetilde{Y}-{Y}|(X^x_{\tau+\cdot}) I_{O^c}]+\hat{\mathbb{E}}[|\widetilde{Y}-{Y}|(X^x_{\tau+\cdot}) I_{O}]
	\leq 2C_Yc(O)\leq 2C_Y\varepsilon.
	\end{align*}
For the right side, since $|Y-\widetilde{Y}|\leq 2C_YI_{D^c}$ and ${D^c}$ is open in $\Omega_n$,
applying Theorem \ref{extended BM strongmarkov1} yields
	\begin{align*}
	\hat{\mathbb{E}}[|\hat{\mathbb{E}}[\widetilde{Y} (X^y_{\cdot})]_{y=X^x_{\tau}}-\hat{\mathbb{E}}[{Y} (X^y_{\cdot})]_{y=X^x_{\tau}}|]
 \leq 2C_Y\hat{\mathbb{E}}[\hat{\mathbb{E}}[ I_{{D^c}}(X^y_{\cdot})]_{y=X^x_{\tau}}]
	=2C_Y\hat{\mathbb{E}}[\hat{\mathbb{E}}_{\tau+}[ I_{{D^c}}(X^x_{\tau+\cdot})]]
=2C_Y\hat{\mathbb{E}}[ I_{{D^c}}(X^x_{\tau+\cdot})],
	\end{align*}
which implies that
\begin{align*}
\hat{\mathbb{E}}[|\hat{\mathbb{E}}[\widetilde{Y} (X^y_{\cdot})]_{y=X^x_{\tau}}-\hat{\mathbb{E}}[{Y} (X^y_{\cdot})]_{y=X^x_{\tau}}|]\leq 2C_Y\hat{\mathbb{E}}[ I_{{D^c}}(X^x_{\tau+\cdot})I_{{O^c}}]+2C_Y\hat{\mathbb{E}}[ I_{{D^c}}(X^x_{\tau+\cdot})I_{{O}}]\leq 2C_Y\varepsilon.
\end{align*}
Sine $\varepsilon$ can be arbitrarily small, it follows that $Y(X^x_{\tau+\cdot})\in L^{1,\tau+}_C(\Omega_d)$ and
	$$\hat{\mathbb{E}}_{\tau+}[{Y}(X^x_{\tau+\cdot})]=\hat{\mathbb{E}}[{Y} (X^y_{\cdot})]_{y=X^x_{\tau}}.$$
	
{\it  2 Unbounded case.}  Define $Y_N=(Y\wedge N)\vee(-N)$ for each $N\geq 1$. By Step 1, we have
	\begin{equation}\label{234353546545}
	\hat{\mathbb{E}}_{\tau+}[{Y_N}(X^x_{\tau+\cdot})]=\hat{\mathbb{E}}[{Y_N} (X^y_{\cdot})]_{y=X^x_{\tau}}.
	\end{equation}
For the left side, we have that
	\begin{align*}
	\hat{\mathbb{E}}[|{Y_N}(X^x_{\tau+\cdot})-{Y}(X^x_{\tau+\cdot})|]
	=\hat{\mathbb{E}}[|{Y_N}-Y|(X^x_{\tau+\cdot})]
	\leq \hat{\mathbb{E}}[(|{Y}|I_{|Y|>N})(X^x_{\tau+\cdot})]	\rightarrow 0,\ \text{as}\ N\rightarrow \infty,
	\end{align*}
which indicates that $Y(X^x_{\tau+\cdot})\in   L^{1,\tau+}_C(\Omega_d)$.
For the right side, it holds that
	\begin{align*}
	\hat{\mathbb{E}}[|\hat{\mathbb{E}}[{Y_N} (X^y_{\cdot})]_{y=X^x_{\tau}}-\hat{\mathbb{E}}[{Y} (X^y_{\cdot})]_{y=X^x_{\tau}}|]
	\leq \hat{\mathbb{E}}[\hat{\mathbb{E}}[(|{Y}|I_{|Y|>N})(X^y_{\cdot})]_{y=X^x_{\tau}}].
	\end{align*}
We claim that for each $N\geq 1$
\begin{align}\label{myq3}
\hat{\mathbb{E}}[(|{Y}|I_{|Y|>N})(X^y_{\cdot})]_{y=X^x_{\tau}}=\hat{\mathbb{E}}_{\tau+}[(|{Y}|I_{|Y|>N})(X^x_{\tau+\cdot})],
\end{align}
	whose proof will be justified in Step 3. Thus, we derive that
\[
\hat{\mathbb{E}}[|\hat{\mathbb{E}}[{Y_N} (X^y_{\cdot})]_{y=X^x_{\tau}}-\hat{\mathbb{E}}[{Y} (X^y_{\cdot})]_{y=X^x_{\tau}}|]\leq
\hat{\mathbb{E}}[(|{Y}|I_{|Y|>N})(X^x_{\tau+\cdot})]\rightarrow 0,\ \text{as}\ N\rightarrow \infty.
\]
Consequently, letting $N\rightarrow\infty$ in (\ref{234353546545}) yields the desired result.

{\it 3 The proof of equation \eqref{myq3}.} For each $I_{\{|y|>N\}}$,
we can choose a sequence $\varphi_k\in C_b(\mathbb{R}^n)$ such that $\varphi_k\uparrow I_{\{|y|>N\}}$. Define
$\bar{Y}^k=(|Y|\wedge k)\varphi_k(Y)$ and it is obvious that $\bar{Y}^k\uparrow |{Y}|I_{\{|Y|>N\}}$. Then Step 1 and Proposition \ref{downward convergence proposition},
Proposition 3.25 (iv) in \cite{HJL} yield (\ref{myq3}). The proof is complete.
\end{proof}

\begin{corollary}\label{belong to LC1 corollary}
If  $Y\in  {L}_C^1(\Omega_n,\mathcal{P}\circ (X^x_{\cdot})^{-1})$,
then \begin{equation}\label{98876444556}	Y(X^x_{\cdot})\in L^{1}_C(\Omega_d).\end{equation} In particular,  $\tau_Q^x,\tau_{\overline{Q}}^x\in  L^{1}_C(\Omega_d).$
	\end{corollary}
\begin{proof}
	Taking $\tau\equiv0$ in Theorem \ref{extended strong markov theorem0}, we get (\ref{98876444556}). From Theorem \ref{exit time quasi continuity lemma}, we have  ${\tau}_Q^{0,1}, {\tau}_{\overline{Q}}^{0,1}\in L_C^1(\Omega_n,\mathcal{P}\circ (X^x_{\cdot})^{-1})$, which ends the proof.
\end{proof}
\begin{proposition}\label{tauQ continuous lemma}
Let $\tau$ be an optional time  such that  $ \tau\leq {\tau}_{Q}^{x}$, q.s. Then
\begin{equation}\label{9897687643435}
{{\tau}_{Q}^{0,1}}={{\tau}_{\overline{Q}}^{0,1}},\ \ \ \ \mathcal{P}\circ (X^x_{\tau+\cdot})^{-1}\text{-q.s.}
\end{equation}
Moreover, ${{\tau}_{Q}^{0,1}}$ and ${{\tau}_{\overline{Q}}^{0,1}}$ both belong to the space  ${L}_C^1(\Omega_n, \mathcal{P}\circ (X^x_{\tau+\cdot})^{-1})$.
\end{proposition}
\begin{proof}
Employing the symbols in the proof of  Theorem \ref{exit times equal lemma},  we can get that
\begin{equation}\label{992987755576}
\hat{\mathbb{E}}[\delta_t(X^x_{\tau+\cdot})]\leq \hat{\mathbb{E}}[f_m(X^x_{\tau+\cdot})]=\mathbb{\hat{E}}[\mathbb{\hat{E}}[f_m(X^y_{\cdot})]_{y=X^x_{{\tau}}}]
\end{equation}
Note that $\hat{\mathbb{E}}[\delta_t(X^y_{\cdot})]=0$  for each $y\in \overline{Q}$.
 We can repeat the analysis  in the proof of Theorem \ref{exit times equal lemma} to obtain the rightside of (\ref{992987755576}) converges to $0$ and this indicates that the equation (\ref{9897687643435}) holds.

Recalling Theorem \ref{exit time quasi continuity lemma},  for any fixed $\varepsilon>0$,  there exists an open set
$O\subset \Omega_n$ such that $c^{\overline{Q}}_2(O)\leq \varepsilon$ and ${{\tau}_{Q}^{0,1}},{{\tau}_{\overline{Q}}^{0,1}}$ are continuous on  $O^c$.
 Note that $c^{y}_2(O)\leq c^{\overline{Q}}_2(O)\leq \varepsilon$ for each $y\in \overline{Q}$. Then using Theorem \ref{extended BM strongmarkov1},  we have that
$$c((X^x_{\tau+\cdot})^{-1}(O))=\hat{\mathbb{E}}[I_{O}(X^x_{\tau+\cdot})]=\hat{\mathbb{E}}[\hat{\mathbb{E}}[I_{O}(X^y_{\cdot})]_{y=X^x_{\tau}}]\leq \varepsilon,$$
which together with
${\tau}_{Q}^{0,1}(X^x_{\tau+\cdot})\leq   {\tau}_{\overline{Q}}^{0,1}(X^x_{\tau+\cdot})\leq {\tau}^x_{\overline{Q}}$
imply that  ${{\tau}_{Q}^{0,1}},{{\tau}_{\overline{Q}}^{0,1}}\in {L}_C^1(\Omega_n, \mathcal{P}\circ (X^x_{\tau+\cdot})^{-1}).$
\end{proof}

The following theorem plays a key role in proving the probabilistic representations.
\begin{theorem}\label{brownian motion DDP for bounded domain}
Assume $\varphi\in C(\partial Q)$ and $f\in C(\overline{Q})$. Let $u(x):=\hat{\mathbb{E}}[\varphi(X^x_{{{\tau}_{Q}^x}})-\int_0^{{{{\tau}_{Q}^x}}}f(X^x_s)ds]$. Then for any optional time $\tau\leq {{\tau}_{Q}^x}$ q.s., we have
\begin{equation}\label{333333}
u(x)=\hat{\mathbb{E}}[u(X^x_{\tau})-\int_0^{\tau}f(X^x_s)ds].
\end{equation}
\end{theorem}
\begin{proof}
By Proposition \ref{tauQ continuous lemma}, for any $\varepsilon>0$ we can choose
 a closed set $D\subset \Omega_n$  such that $c((X^x_{\tau+\cdot})^{-1}(D^c))\leq \varepsilon$ and
${{\tau}_{Q}^{0,1}}$ is continuous on $D$.
Then $Y:=\varphi(B'_{{{\tau}_{Q}^{0,1}}})-\int_{0}^{{{{\tau}_{Q}^{0,1}}}}f(B'_s)ds$ is also continuous on $D$, where $B'$ is the canonical process on $\Omega_n$.
Hence, $Y$ is $\mathcal{P}\circ (X^x_{\tau+\cdot})^{-1}$-quasi-continuous on $\Omega_n$.
For each $k\geq 1$, set $Y_k:=\varphi(B'_{{{\tau}_{Q}^{0,1}}})-\int_{0}^{{{{\tau}_{Q}^{0,1}}}\wedge k}f(B'_s)ds$, which is
 also $\mathcal{P}\circ (X^x_{\tau+\cdot})^{-1}$-quasi-continuous on $\Omega_n$. Thus $Y^k$ belongs to ${L}_C^{1}(\Omega_n,\mathcal{P}\circ (X^x_{\tau+\cdot})^{-1})$ by Theorem \ref{LG characteriazation theorem}.

Note that for each $k\geq 1$
\begin{align*}
\hat{\mathbb{E}}[|Y-Y_k|(X^x_{\tau+\cdot})]=\hat{\mathbb{E}}[|\int_{{{{\tau}_{Q}^{0,1}}}\wedge k}^{{{{\tau}_{Q}^{0,1}}}}f(B'_s)ds|(X^x_{\tau+\cdot})]
=\hat{\mathbb{E}}[|\int_{({{{\tau}_{Q}^{x}}}-{\tau})\wedge k}^{{{{\tau}_{Q}^{x}}}-{{\tau}}}f(X^x_{{\tau}+s})ds|]
\leq C_f \hat{\mathbb{E}}[{{{{\tau}_{Q}^{x}}}-{{\tau}}}-({{{{\tau}_{Q}^{x}}}-{{\tau}}})\wedge k].
\end{align*}
Then by Lemma \ref{square stopping time lemma}, we have that\begin{align*}
\hat{\mathbb{E}}[|Y-Y_k|(X^x_{\tau+\cdot})]&\leq C_f \hat{\mathbb{E}}[({{{{\tau}_{Q}^{x}}}-{{\tau}}}- k)I_{\{{{{{\tau}_{Q}^{x}}}-{{\tau}}}>k \}}]
\leq C_f {\hat{\mathbb{E}}[{\tau}_{Q}^{x}I_{\{{\tau}_{Q}^{x}>k \}}]}
\leq C_f \frac{\hat{\mathbb{E}}[({\tau}_{Q}^{x})^2]}{k}
\rightarrow 0, \ \text{as $k\rightarrow\infty$}.
\end{align*}
This implies that  $Y\in {L}_C^{1}(\Omega_n,\mathcal{P}\circ (X^x_{\tau+\cdot})^{-1})$.

Now applying Theorem \ref{extended strong markov theorem0}, we  obtain that
\begin{equation*}
\begin{split}
\hat{\mathbb{E}}_{\tau+}[\varphi(X^x_{{{\tau}_{Q}^x}})-\int_{{{\tau}}}^{{{{\tau}_{Q}^x}}}f(X^x_s)ds] =\hat{\mathbb{E}}_{\tau+}[Y(X^x_{\tau+\cdot})]
=\hat{\mathbb{E}}[Y(X^y_{\cdot})]_{y=X^x_{\tau}}=\hat{\mathbb{E}}[\varphi(X^y_{{{\tau}_{Q}^y}})-\int_{{{0}}}^{{{{\tau}_{Q}^y}}}f(X^y_s)ds]_{y=X^x_{\tau}}
=u(X^x_{\tau}).
\end{split}
\end{equation*}
Therefore, we derive that
\begin{equation*}
\begin{split}
u(x)
=\hat{\mathbb{E}}[\hat{\mathbb{E}}_{\tau+}[\varphi(X^x_{{{\tau}_{Q}^x}})-\int_{{{\tau}}}^{{{{\tau}_{Q}^x}}}f(X^x_s)ds]-\int_0^{{{\tau}}}f(X^x_s)ds]
=\hat{\mathbb{E}}[u(X^x_{\tau})-\int_0^{{{\tau}}}f(X^x_s)ds],
\end{split}
\end{equation*}
which ends the proof.
\end{proof}

Now we are ready to state our main result of this section, concerning a probabilistic representation for the viscosity solutions to fully nonlinear PDEs. For the definition and properties of  viscosity solutions, we refer the reader to \cite{CC,CIL,IL}.
\begin{theorem}\label{viscosity solution theorem}
Assume that  $\varphi\in C(\partial Q)$ and $f\in C(\overline{Q})$.
Then $u(x):=\hat{\mathbb{E}}[\varphi(X^x_{{\tau}^{x}_Q})-\int_0^{{\tau}^{x}_Q}f(X^x_{s})ds]$ is the $C(\overline{Q})$-continuous viscosity solution of
\begin{equation}\label{G elliptic PDE}
\begin{cases}
G(\sigma(x)^TD^2u(x)\sigma(x)+H(Du(x)),\sigma(x)^TDu(x))+  \langle b(x),Du(x)\rangle=f(x),\ x\in Q,\\
u(x)=\varphi(x),\ x\in \partial Q,
\end{cases}
\end{equation}
where $H_{ij}(Du):=2\langle Du,h_{ij}\rangle$, $1\leq i,j\leq d$.
\end{theorem}
\begin{proof}
The uniqueness of  viscosity solutions can be found in \cite{CIL}.
 The proof shall be divided into two steps.

{\it 1 The continuity.} We first consider the case that $\varphi\in C_{b.Lip}(\partial Q)$ and $f\in C_{b.Lip}(\overline{Q})$.

Assume $x_k\rightarrow x$ on $\overline{Q}$. By the sub-linearity of $\hat{\mathbb{E}}$, we have
\begin{equation}\label{eq. 21}
\begin{split}
|u(x)-u(x_k)|
\leq \hat{\mathbb{E}}[|\varphi(X^{x}_{{\tau}^{x}_Q})-\varphi(X^{x_k}_{{\tau}^{x_k}_Q})|]
+\hat{\mathbb{E}}[|\int_0^{{\tau}^{x}_Q}f(X^{x}_{s})ds-\int_0^{{\tau}^{x_k}_Q}f(X^{x_k}_{s})ds|].
\end{split}
\end{equation}
Then we just need to prove that the above two terms in equation (\ref{eq. 21}) converge to 0, as $k\rightarrow \infty$.

For each $T>0$ and $\varepsilon>0$, we can decompose the first term into three parts as follows:
\begin{align}\label{myq4}
\begin{split}
&\hat{\mathbb{E}}[|\varphi(X^{x}_{{\tau}^{x}_Q})-\varphi(X^{x_k}_{{\tau}^{x_k}_Q})|]
\\&\leq \hat{\mathbb{E}}[|\varphi(X^{x}_{{\tau}^{x}_Q})-\varphi(X^{x_k}_{{\tau}^{x_k}_Q})
|I_{\{|{\tau}^{x}_Q(\omega)-{\tau}^{x_k}_Q(\omega)|<
\varepsilon\}}I_{\{{\tau}^{x}_{{Q}}\vee {\tau}^{x_k}_{{Q}}\leq T\}}]\\ &\ \
+\hat{\mathbb{E}}[|\varphi(X^x_{{\tau}^{x}_Q})-\varphi(X^{x_k}_{{\tau}^{x_k}_Q})
|I_{\{|{\tau}^{x}_Q(\omega)-{\tau}^{x_k}_Q(\omega)|<
\varepsilon\}}I_{\{{\tau}^{x}_{{Q}}\vee {\tau}^{x_k}_{{Q}}> T\}}] +\hat{\mathbb{E}}[|\varphi(X^{x}_{{\tau}^{x}_Q})-\varphi(X^{x_k}_{{\tau}^{x_k}_Q})
|I_{\{|{\tau}^{x}_Q(\omega)-{\tau}^{x_k}_Q(\omega)|\geq \varepsilon\}}]\\
&\leq \hat{\mathbb{E}}[|\varphi(X^{x}_{{\tau}^{x}_Q})-\varphi(X^{x_k}_{\tau^{x_k}_Q})
|I_{\{|{\tau}^{x}_Q(\omega)-{\tau}^{x_k}_Q(\omega)|<
\varepsilon\}}I_{\{{\tau}^{x}_{{Q}}\vee {\tau}^{x_k}_{{Q}}\leq T\}}]
+2C_\varphi\hat{\mathbb{E}}[I_{\{{\tau}^{x}_{{Q}}\vee {\tau}^{x_k}_{{Q}}> T\}}]+2C_\varphi\hat{\mathbb{E}}[I_{\{|{\tau}^{x}_Q(\omega)-{\tau}^{x_k}_Q(\omega)|\geq \varepsilon\}}]\\
&=:I_1^{k,\varepsilon,T}+I_2^{k,T}+I_3^{k,\varepsilon}.
\end{split}
\end{align}
Now we shall deal with the three parts separately.
For $I_1^{k,\varepsilon,T}$, by a direct calculation we have that
\begin{align*}
I_1^{k,\varepsilon,T}
&\leq \hat{\mathbb{E}}[|\varphi(X^{x}_{{\tau}^{x}_Q})-\varphi(X^{x}_{{\tau}^{x_k}_Q})
|I_{\{|{\tau}^{x_k}_Q(\omega)-{\tau}^{x}_Q(\omega)|<
	\varepsilon\}}I_{\{{\tau}^{x}_{{Q}}\vee {\tau}^{x_k}_{{Q}}\leq T\}}]
\\
&\ \ \ \ \ \ \ \ +\hat{\mathbb{E}}[|\varphi(X^{x}_{{\tau}^{x_k}_Q})-\varphi(X^{x_k}_{{\tau}^{x_k}_Q})
|I_{\{|{\tau}^{x}_Q(\omega)-{\tau}^{x_k}_Q(\omega)|<
\varepsilon\}}I_{\{{\tau}^{x}_{{Q}}\vee {\tau}^{x_k}_{{Q}}\leq T\}}]\\
&\leq L_{\varphi}\hat{\mathbb{E}}[\sup_{\substack{t,s\in [0,T]\\0\leq |t-s|\leq\varepsilon }}|X^{x}_{t}-X^{x}_{s}|]+L_{\varphi}\hat{\mathbb{E}}[\sup_{  t\in [0,T]}|X^{x}_{t}-X^{x_k}_{t}
|].
\end{align*}
For each integer $\rho\geq 1$, denote $t^{\rho}_i=\frac{i}{\rho}T$, $i=0, \ldots, \rho$.
Then one can easily check that
\[
\sup_{\substack{t,s\in [0,T]\\0\leq |t-s|\leq\varepsilon }}| X^{x}_t- X^{x}_s|\leq 3\sup_{i}\sup_{s\in[t^{\rho}_i,t^{\rho}_{i+1}]}| X^{x}_{t^{\rho}_i}- X^{x}_s|,
\]
whenever $\varepsilon\leq \frac{T}{\rho}.$ Thus by a standard argument we  can find some generic constant $C_{T}>0$ (which may vary from line to line) independent of $k,\varepsilon$ so that,
for each $\varepsilon\leq \frac{T}{\rho}$,
\begin{align*}
\hat{\mathbb{E}}[\sup_{\substack{t,s\in [0,T]\\0\leq |t-s|\leq\varepsilon }}| X^{x}_t- X^{x}_s|^{4}]
\leq 3^4\sum\limits_{i=0}^{\rho-1}\hat{\mathbb{E}}[\sup_{s\in[t^{\rho}_i,t^{\rho}_{i+1}]}| X^{x}_{t^{\rho}_i}- X^{x}_s|^{4}]\leq \frac{C_T}{\rho}.
\end{align*}
Moreover, it holds that
\[
\hat{\mathbb{E}}[\sup_{  t\in [0,T]}|X^{x}_{t}-X^{x_k}_{t}
|]\leq C_T|x-x_k|.
\]
Consequently, we obtain that for each  $\rho$
\[
I_1^{k,\varepsilon,T}\leq C_TL_{\varphi}(\frac{1}{\rho^{\frac{1}{4}}}+|x-x_k|), \ \text{for each}\ \varepsilon\leq \frac{T}{\rho},
\]
which indicates that
$\limsup\limits_{k,\varepsilon\rightarrow 0}I_1^{k,\varepsilon,T}=0$ for each $T>0$.
\\
For   $I_2^{k,T}$, it follows from Lemma \ref{stopping time lemma}  and Markov's inequality that\[
I_2^{k,T}\leq 2C_{\varphi}\hat{\mathbb{E}}[I_{\{{\tau}^{x}_{{Q}}+ {\tau}^{x_k}_{{Q}}> T\}}]\leq 2C_{\varphi}\{ \hat{\mathbb{E}}[I_{\{{\tau}^{x}_{{Q}}> \frac{T}{2}\}}]+\hat{\mathbb{E}}[I_{\{{\tau}^{x_k}_{{Q}}> \frac{T}{2}\}}]\}
\leq \frac{8CC_{\varphi}}{T}, \ \text{for each} \ T>0.\]
 For $I_3^{k,\varepsilon}$, it follows from Lemma \ref{tau continuous lemma} that
$
\limsup\limits_{k\rightarrow\infty}I_3^{k,\varepsilon}=0
$ for each $\varepsilon>0.$

By the above analysis, letting $k,\varepsilon\rightarrow 0$ and then sending $T\rightarrow\infty$ in equation \eqref{myq4}
yield that
\[
\limsup\limits_{k\rightarrow\infty}\hat{\mathbb{E}}[|\varphi(X^{x}_{{\tau}^{x}_Q})-\varphi(X^{x_k}_{{\tau}^{x_k}_Q})|]=0.
\]

Now we consider the second term in equation (\ref{eq. 21}). Since $f$ is bounded and Lipschitz continuous on $\overline{Q}$, we have
\begin{align*}
&\hat{\mathbb{E}}[|\int_0^{{\tau}^{x}_Q}f(X^x_{s})ds-\int_0^{{\tau}^{x_k}_Q}f(X^{x_k}_{s})ds|]\\
&\leq \hat{\mathbb{E}}[|\int_0^{{\tau}^{x}_Q\wedge {\tau}^{x_k}_Q}(f(X^{x}_{s})-f(X^{x_k}_{s}))ds|]+
2C_f\hat{\mathbb{E}}[{{\tau}^{x}_Q\vee {\tau}^{x_k}_Q}-{{\tau}^{x}_Q\wedge {\tau}^{x_k}_Q}]\\
&\leq \hat{\mathbb{E}}[|\int_0^{{\tau}^{x}_Q\wedge {\tau}^{x_k}_Q}(f(X^{x}_{s})-f(X^{x_k}_{s}))ds|I_{\{{{\tau}^{x}_Q\wedge {\tau}^{x_k}_Q}\leq T\}}]+2C_f\hat{\mathbb{E}}[{{\tau}^{x}_Q\wedge {\tau}^{x_k}_Q}I_{\{{{\tau}^{x}_Q\wedge {\tau}^{x_k}_Q}> T\}}]+
2C_f\hat{\mathbb{E}}[{{\tau}^{x}_Q\vee {\tau}^{x_k}_Q}-{{\tau}^{x}_Q\wedge {\tau}^{x_k}_Q}]\\
&\leq L_f T\hat{\mathbb{E}}[\sup_{  t\in [0,T]}|X^{x}_{t}-X^{x_k}_{t}
|]+2C_f\hat{\mathbb{E}}[{{\tau}^{x}_Q\wedge {\tau}^{x_k}_Q}I_{\{{{\tau}^{x}_Q\wedge {\tau}^{x_k}_Q}> T\}}]+
2C_f\hat{\mathbb{E}}[{{\tau}^{x}_Q\vee {\tau}^{x_k}_Q}-{{\tau}^{x}_Q\wedge {\tau}^{x_k}_Q}].
\end{align*}
For any $\delta>0$, by first letting $T$ large enough such that the second term is smaller than $2C_f\delta$, then letting $k\rightarrow\infty$, we deduce that
$$
\limsup\limits_{k\rightarrow\infty}\hat{\mathbb{E}}[|\int_0^{{\tau}^{x}_Q}f(X^x_{s})ds-\int_0^{{\tau}^{x_k}_Q}f(X^{x_k}_{s})ds|]\leq 2C_f\delta,$$
which implies
$$
\limsup\limits_{k\rightarrow\infty}\hat{\mathbb{E}}[|\int_0^{{\tau}^{x}_Q}f(X^x_{s})ds-\int_0^{{\tau}^{x_k}_Q}f(X^{x_k}_{s})ds|]=0.$$
Therefore, we obtain the continuity of  $u$ on $\overline{Q}$.

For the general case that $\varphi\in C(\partial Q)$ and $f\in C(\overline{Q})$,
we could find a sequence of bounded and Lipschitz functions $\varphi_n\in C(\partial Q)$ and $f_n\in C(\overline{Q})$ such that $\varphi_n$ and $f_n$  converge uniformly to
$\varphi$ and $f$, respectively. Then $u_n$ converges to $u$ uniformly in $\overline{Q}$ and this implies the desired result.

{\it 2 Viscosity solution  property.}
We just   prove the viscosity sub-solution case, since another case can be proved in a similar way.
Assume that $u$ does not satisfy the viscosity sub-solution property. Then there exists a test function $\phi\in C^2(\overline{Q})$ such that
$\phi\geq u$ on $Q$, $\phi(x_0)=u(x_0)$ for some point  $x_0\in Q$ and
$$
G(\sigma(x_0)^TD^2\phi(x_0)\sigma(x_0)+H(D\phi(x_0)),\sigma(x_0)^TD\phi(x_0))+  \langle b(x_0),D\phi(x_0)\rangle<f(x_0).
$$
By the continuity, we can find an open ball $U(x_0,{\delta_0})\subset Q$ for some   $\delta_0>0$ such that
$$G(\sigma(x)^TD^2\phi(x)\sigma(x)+H(D\phi(x)),\sigma(x)^TD\phi(x))+  \langle b(x),D\phi(x)\rangle <f(x),  \ \text{for all} \ x\in U(x_0,{\delta_0}).$$
 Moreover, ${\tau}^{x_0}_{U(x_0,{\delta_0})}>0$ for q.s. $\omega$.

Set $\Upsilon(x):=G(\sigma(x)^TD^2\phi(x)\sigma(x)+H(D\phi(x)),\sigma(x)^TD\phi(x))$.
Applying It\^{o} formula to $\phi$, we have
\begin{align*}
&\phi(X^{x_0}_{{\tau}^{x_0}_{U(x_0,{\delta_0})}\wedge t})-\phi(x_0)-\int_0^{{\tau}^{x_0}_{U(x_0,{\delta_0})}\wedge t}\langle D\phi(X^{x_0}_s), b(X^{x_0}_s)\rangle ds\\
&=\int_0^{{\tau}^{x_0}_{U(x_0,{\delta_0})}\wedge t}\langle D\phi(X^{x_0}_s),\sigma(X^{x_0}_s) dB_s\rangle+\frac12\int_0^{{\tau}^{x_0}_{U(x_0,{\delta_0})}\wedge t}\langle \sigma^T(X^{x_0}_s)D^2\phi(X^{x_0}_s)\sigma(X^{x_0}_s)+H(D\phi(X^{x_0}_s)), d\langle B\rangle_s\rangle\\
&=\int_0^{{\tau}^{x_0}_{U(x_0,{\delta_0})}\wedge t}\Upsilon(X^{x_0}_s)ds+M_{{\tau}^{x_0}_{U(x_0,{\delta_0})}\wedge t},
\end{align*}
where $M$ is a $G$-martingale   (Proposition \ref{martingale proposition}) and given by\begin{align*}
M_t=\int_0^{ t}\langle D\phi(X^{x_0}_s),\sigma(X^{x_0}_s) dB_s\rangle+\frac12\int_0^{ t}\langle\sigma^T(X^{x_0}_s)D^2\phi(X^{x_0}_s)\sigma(X^{x_0}_s)+H(D\phi(X^{x_0}_s)), d\langle B\rangle_s\rangle
  -\int_0^{t}\Upsilon(X^{x_0}_s)ds.
\end{align*}
That is,
\begin{align*}
M_{{\tau}^{x_0}_{U(x_0,{\delta_0})}\wedge t}+\phi(x_0)=\phi(X^{x_0}_{{\tau}^{x_0}_{U(x_0,{\delta_0})}\wedge t})-\int_0^{{\tau}^{x_0}_{U(x_0,{\delta_0})}\wedge t}\langle D\phi(X^{x_0}_s), b(X^{x_0}_s)\rangle ds-\int_0^{{\tau}^{x_0}_{U(x_0,{\delta_0})}\wedge t}\Upsilon(X^{x_0}_s)ds.
\end{align*}
Taking expectation on both sides and then using     the optional sampling theorem for $G$-martingales (see  Theorem 48 in \cite{HP1}), we get that
\begin{equation*}
\begin{split}
\phi(x_0)=\hat{\mathbb{E}}[\phi(X^{x_0}_{{\tau}^{x_0}_{U(x_0,{\delta_0})}\wedge t})-\int_0^{{\tau}^{x_0}_{U(x_0,{\delta_0})}\wedge t}(\Upsilon(X^{x_0}_s)+\langle D\phi(X^{x_0}_s), b(X^{x_0}_s)\rangle) ds].
\end{split}
\end{equation*}
Recalling Lemma \ref{square stopping time lemma}, we have that\begin{align*}
\hat{\mathbb{E}}[|\phi(X^{x_0}_{{\tau}^{x_0}_{U(x_0,{\delta_0})}\wedge t})-\phi(X^{x_0}_{{\tau}^{x_0}_{U(x_0,{\delta_0})}})|]&\leq 2C_{\phi}\hat{\mathbb{E}}[I_{\{{\tau}^{x_0}_{U(x_0,{\delta_0})}\geq t\}}]\rightarrow 0, \ \text{as}\ t\rightarrow\infty,\\
\hat{\mathbb{E}}[|\int_0^{{\tau}^{x_0}_{U(x_0,{\delta_0})}\wedge t}\psi(X^{x_0}_s)ds-\int_0^{{\tau}^{x_0}_{U(x_0,{\delta_0})}}\psi(X^{x_0}_s)ds|]&\leq 2C_{\psi}\hat{\mathbb{E}}[{\tau}^{x_0}_{U(x_0,{\delta_0})}I_{\{{\tau}^{x_0}_{U(x_0,{\delta_0})}\geq t\}}]\rightarrow 0,\ \text{as}\ t\rightarrow\infty,
\end{align*}
for $\psi:=\Upsilon+\langle D\phi, b\rangle.$
Therefore, it follows that
\begin{equation*}\label{98767833333575}
\begin{split}
\phi(x_0)&=\hat{\mathbb{E}}[\phi(X^{x_0}_{{\tau}^{x_0}_{U(x_0,{\delta_0})}})-\int_0^{{\tau}^{x_0}_{U(x_0,{\delta_0})}}(\Upsilon(X^{x_0}_s)+\langle D\phi(X^{x_0}_s), b(X^{x_0}_s)\rangle) ds]\\
&=\hat{\mathbb{E}}^{x_0}_2[\phi(B'_{{\tau}^{0,1}_{U(x_0,{\delta_0})}})
-\int_0^{{\tau}^{0,1}_{U(x_0,{\delta_0})}}(\Upsilon(B'_s)+\langle D\phi(B'_s),b(B'_s)\rangle) ds].
\end{split}
\end{equation*}

Note that ${{\tau}^{0,1}_{U(x_0,{\delta_0})}}\in L_C^1(\Omega_n,\mathcal{P}\circ (X_\cdot^{x_0})^{-1})$ by Lemma \ref{exit time quasi continuity lemma}.
Then a similar analysis as in the first part in the proof of Theorem \ref{brownian motion DDP for bounded domain} gives
$\phi(B'_{{\tau}^{0,1}_{U(x_0,{\delta_0})}})
-\int_0^{{\tau}^{0,1}_{U(x_0,{\delta_0})}}\widetilde{\psi}(B'_s)ds\in L_C^1(\Omega_n,\mathcal{P}\circ (X_\cdot^{x_0})^{-1})$ for each $\widetilde{\psi}\in C(\overline{Q})$. Thus in sprit of Remark \ref{remark on tightness guarantee maximum and on closure} and the fact that $\Upsilon+\langle D\phi, b\rangle< f$ on $U(x_0,{\delta_0})$,
 we conclude that
\begin{equation*}
\begin{split}
\phi(x_0)&=\max_{P\in \mathcal{P}}E_{P\circ (X^{x_0}_\cdot)^{-1}}[\phi(B'_{{\tau}^{0,1}_{U(x_0,{\delta_0})}})-\int_0^{{\tau}^{0,1}_{U(x_0,{\delta_0})}}(\Upsilon(B'_s)+\langle D\phi(B'_s),b(B'_s)\rangle) ds]\\
&>\max_{P\in \mathcal{P}}E_{P\circ (X^{x_0}_\cdot)^{-1}}[\phi(B'_{{\tau}^{0,1}_{U(x_0,{\delta_0})}})-\int_0^{{\tau}^{0,1}_{U(x_0,{\delta_0})}}f(B'_s)ds]\\
&=\hat{\mathbb{E}}[\phi(X^{x_0}_{{\tau}^{x_0}_{U(x_0,{\delta_0})}})-\int_0^{{\tau}^{x_0}_{U(x_0,{\delta_0})}}f(X^{x_0}_s)ds].
\end{split}
\end{equation*}
However, by Theorem \ref{brownian motion DDP for bounded domain}, we get that
$$
u(x_0)=\hat{\mathbb{E}}[u(X^{x_0}_{{\tau}^{x_0}_{U(x_0,{\delta_0})}})-\int_0^{{\tau}^{x_0}_{U(x_0,{\delta_0})}}f(X^{x_0}_s)ds]\leq \hat{\mathbb{E}}[\phi(X^{x_0}_{{\tau}^{x_0}_{U(x_0,{\delta_0})}})-\int_0^{{\tau}^{x_0}_{U(x_0,{\delta_0})}}f(X^{x_0}_s)ds]<\phi(x_0)=u(x_0),
$$
which is a contradiction. The proof is complete.
\end{proof}

\begin{corollary}\label{supremum realization}
	Assume that  $\varphi\in C(\partial Q)$ and $f\in C(\overline{Q})$. For $u$ defined as the above theorem, we have
	\begin{equation}u(x)=\max_{P\in\mathcal{P}}E_P[\varphi(X^x_{{\tau}^{x}_Q})-\int_0^{{\tau}^{x}_Q}f(X^x_{s})ds]\end{equation}
\end{corollary}
\begin{proof}
According to the proof of Theorem \ref{brownian motion DDP for bounded domain}, we have
$\varphi(B'_{{\tau}^{0,1}_{Q}})
-\int_0^{{\tau}^{0,1}_{Q}}f(B'_s)ds\in L_C^1(\Omega_n,\mathcal{P}\circ (X_\cdot^{x})^{-1})$. Then Corollary  \ref{belong to LC1 corollary}  implies  $\varphi(X^x_{{\tau}^{x}_Q})-\int_0^{{\tau}^{x}_Q}f(X^x_{s})ds\in L_C^1(\Omega_d)$, and the desired result now follows from Remark \ref{remark on tightness guarantee maximum and on closure}.
	\end{proof}

The following result is a direct conclusion of   Theorem \ref{viscosity solution theorem}.
\begin{corollary}\label{viscosity solution theorem2}
Assume $\varphi,f$ satisfy the same condition  as above. Then  $u(x):=-\hat{\mathbb{E}}[-\varphi(X^x_{{\tau}^{x}_Q})+\int_0^{{\tau}^{x}_Q}f(X^x_{s})ds]$ is the $C(\overline{Q})$-continuous viscosity solution of
\begin{equation}\label{-G}
\begin{cases}
-G(-\sigma(x)^TD^2u(x)\sigma(x)-H(Du(x)),-\sigma(x)^TDu(x))+ \langle b(x),Du(x)\rangle=f(x),\ x\in Q,\\
u(x)=\varphi(x),\ x\in \partial Q.
\end{cases}
\end{equation}
\end{corollary}
\begin{proof}
	The proof is immediate from  Theorem \ref{viscosity solution theorem} by taking $\tilde{\varphi}:=-\varphi$, $\tilde{f}:=-f$ and $\tilde{u}(x):=-u(x)$.
\end{proof}

\appendix
\renewcommand\thesection{Appendix: The  proof of Theorem \ref{generalized G-BM decomposition}}
\section{ }
\renewcommand\thesection{A}
Now we are ready to state the  proof of Theorem \ref{generalized G-BM decomposition}.

\begin{proof}
 It suffices to  prove the decomposition on $[0,1]$.
The main idea is based on the technique of Doob-Meyer's decomposition,  see, e.g., \cite{KS}.
For  simplicity, we omit the superscript $P$ on $A^P$ and $M^P$. The proof shall be divided into two steps.

{\it 1 Existence.} For each integer $m>0$, let $\Pi_m:=\{t^m_j=\frac{j}{2^m}:j\leq {2^m} \}$ be a partition of $[0,1]$. Denote $\Pi:=\bigcup_{m=1}^{\infty}\Pi_m$.
	Then we define
	$$
	A^m_{0}:=0,\ \ \ \ A^m_{t^m_{j}}:=A^m_{t^m_{j-1}}+E_P[B_{t^m_{j}}-B_{t^m_{j-1}}|\mathcal{F}_{t^m_{j-1}}]=\sum_{k=0}^{j-1}E_P[B_{t^m_{k+1}}- B_{t^m_{k}}|\mathcal{F}_{t^m_{k}}], \ \forall j\leq 2^m.
	$$
	For each $m$, a direct calculation shows that $B_t-A_t^m$ is  a martingale on $\Pi_m$.
	By Proposition \ref{properties of generalized G-BM} and Lemma \ref{STZ lemma for generalized G},
 we can find some constant $L$ independent of $m$ so that $E_P[|A^m_{1}|^2]\leq L$.
Thus there exists a subsequence, still denote by $A^m_{1}$, such that
$A^m_{1}$ converges to $A_1$ in $L^2(\Omega,\mathcal{F}_1,P;\mathbb{R}^d)$ weakly.
Note that $A_t^m=B_t-E_P[ B_1-A_1^m|\mathcal{F}_{t}]$ for each $t\in\Pi_m$.
We define $A_t$ as the right-continuous modification of $B_t-E_P[ B_1-A_1|\mathcal{F}_{t}]$, $t\in[0,1]$. Then it is obvious that
	$A^m_{t}$ converges to  $A_t$ in $L^2(\Omega,\mathcal{F}_1,P;\mathbb{R}^d)$ weakly for each $t\in \Pi$.

From Proposition \ref{properties of generalized G-BM} and  Lemma \ref{STZ lemma for generalized G},
we derive that for $s<t\in\Pi_m$,
	$$
	\langle p,A^m_t-A^m_s\rangle\leq g(p)(t-s), \ \forall p\in\mathbb{R}^d.
	$$
It follows that for any $s,t\in \Pi$ such that $s<t$ and $0\leq \xi\in L^2(\Omega,\mathcal{F}_1,P)$,
	\begin{equation*}
	E_P[\xi(\langle p,A_t\rangle-\langle p, A_s\rangle)]=\lim_{n\rightarrow\infty}E_P[\xi(\langle p,A^n_t\rangle-\langle p,A^n_s\rangle)]\leq E_P[\xi g(p)(t-s)],
	\end{equation*}	
which implies
	$\langle p,A_t\rangle-\langle p, A_s\rangle\leq g(p)(t-s)\  P\text{-a.s.},
$   for each $t\geq s\in \Pi$.
Note that $A$ is  right continuous, we get that
	\begin{equation}\label{3244444444444323}
	\langle p,A_t\rangle-\langle p, A_s\rangle\leq g(p)(t-s),\ \ \ \ \text{for each}\ t> s\in [0,1], \ P\text{-a.s.}
	\end{equation}
Therefore, it holds that   $
|\langle p,A_t\rangle-\langle p, A_s\rangle|\leq [g(p)\vee g(-p)](t-s),  P\text{-a.s.}$	
and thus $\langle p,A_t\rangle$ is absolutely continuous on $[0,1]$. Dividing $t-s>0$ on both sides of (\ref{3244444444444323}) and letting $s\rightarrow t$, we obtain
\[\langle p,\frac{dA_t}{dt}\rangle\leq g(p), \ \ \ \ \ {a.e.\  t},\  P\text{-a.s.}\]
Consequently, the decomposition is obtained after we define the martingale
	$$M_t:=B_t-A_t=E_P[B_1-A_1|\mathcal{F}_{t}].$$
	
{\it 2	The estimate of $\langle M\rangle ^P$}.
For each $s<t\leq 1$, we take $\Pi_m:=\{t^m_k=s+\frac{k}{m}(t-s):k\leq {m} \}$ as the partition of $[s,t]$, $m\geq 1$.
For any given $A\in \mathbb{S}(d)$, we have
	\begin{align*}
	&E_P[(\sum_{k=0}^{m-1} \langle A,(B_{t^m_{k+1}}-B_{t^m_{k}})(B_{t^m_{k+1}}-B_{t^m_{k}})^T\rangle-G_1(A)(t-s))^+]\\
	& \leq \hat{\mathbb{E}}[(\sum_{k=0}^{m-1}\langle A,(B_{t^m_{k+1}}-B_{t^m_{k}})(B_{t^m_{k+1}}-B_{t^m_{k}})^T\rangle-G_1(A)(t-s))^+]\\
	&=\sup_{(\gamma,\mu) \in \mathcal{A}^{\Theta}}E_{P^0}[ (\sum_{k=0}^{m-1}\langle A,(\int_{t^m_{k}}^{t^m_{k+1}}\gamma_rdW_r+\int_{t^m_{k}}^{t^m_{k+1}}\mu_rdr)(\int_{t^m_{k}}^{t^m_{k+1}}\gamma_rdW_r+\int_{t^m_{k}}^{t^m_{k+1}}\mu_rdr)^T\rangle-G_1(A)(t-s))^+    ]\\
	&\leq I_1+I_2+I_3,
	\end{align*}
with
	\begin{equation}
	\begin{split}
	&I_1:=\sup_{(\gamma,\mu) \in \mathcal{A}^{\Theta}}E_{P^0}[ (\sum_{k=0}^{m-1}\langle A,(\int_{t^m_{k}}^{t^m_{k+1}}\gamma_rdW_r)(\int_{t^m_{k}}^{t^m_{k+1}}\gamma_rdW_rdr)^T\rangle-G_1(A)(t-s))^+ ],\\
	&I_2:=\sup_{(\gamma,\mu) \in \mathcal{A}^{\Theta}}E_{P^0}[ (\sum_{k=0}^{m-1}\langle A,(\int_{t^m_{k}}^{t^m_{k+1}}\mu_rdr)(\int_{t^m_{k}}^{t^m_{k+1}}\mu_rdr)^T\rangle)^+ ],\\
	&I_3:=2\sup_{(\gamma,\mu) \in \mathcal{A}^{\Theta}}E_{P^0}[ (\sum_{k=0}^{m-1}\langle A,(\int_{t^m_{k}}^{t^m_{k+1}}\gamma_rdW_r)(\int_{t^m_{k}}^{t^m_{k+1}}\mu_rdr)^T\rangle)^+ ].
	\end{split}
	\end{equation}

Next we shall deal with the above three terms separately.
For the $I_2$ term,   a direct calculation gives that, for  some constant $L_1$ depending on  $\Sigma$,
	\begin{align*}
	E_{P^0}[ (\sum_{k=0}^{m-1}\langle A,(\int_{t^m_{k}}^{t^m_{k+1}}\mu_rdr)(\int_{t^m_{k}}^{t^m_{k+1}}\mu_rdr)^T\rangle)^+ ]\leq E_{P^0}[ \sum_{k=0}^{m-1}|\langle A,(\int_{t^m_{k}}^{t^m_{k+1}}\mu_rdr)(\int_{t^m_{k}}^{t^m_{k+1}}\mu_rdr)^T\rangle| ]& \leq 	\frac{L_1|A|}{m}(t-s)^2.
	\end{align*}
By a similar analysis,  we can find  some constant $L_2$   depending on  $\Sigma$ and $\Gamma$ such that	
$$I_3\leq \frac{L_2|A|}{\sqrt{m}}(t-s)^\frac32.$$
  Now we consider the $I_1$ term. By the definition of $G_1$, we derive that
	\begin{align*}
	I_1& \leq E_{P^0}[ (\langle A,\int_s^t\gamma_r\gamma_r^Tdr\rangle-G_1(A)(t-s))^+ ] \\
	&\ \ \ \ +E_{P^0}[ (\sum_{k=0}^{m-1}\langle A,(\int_{t^m_{k}}^{t^m_{k+1}}\gamma_sdW_s)(\int_{t^m_{k}}^{t^m_{k+1}}\gamma_rdW_r)^T -\int_{t^m_{k}}^{t^m_{k+1}}\gamma_r\gamma_r^Tdr\rangle)^+ ]\\
	& \leq   E_{P^0}[ (\sum_{k=0}^{m-1}\langle A,(\int_{t^m_{k}}^{t^m_{k+1}}\gamma_rdW_r)(\int_{t^m_{k}}^{t^m_{k+1}}\gamma_rdW_r)^T-\int_{t^m_{k}}^{t^m_{k+1}}\gamma_r\gamma_r^Tdr\rangle)^+ ]\\
	&\leq E_{P^0}[(\sum_{k=0}^{m-1}\langle A,(\int_{t^m_{k}}^{t^m_{k+1}}\gamma_rdW_r)(\int_{t^m_{k}}^{t^m_{k+1}}\gamma_rdW_r)^T-\int_{t^m_{k}}^{t^m_{k+1}}\gamma_r\gamma_r^Tdr\rangle)^2]^\frac12.
	\end{align*}
	Note that $\langle A,(\int_{t^m_{k}}^{t}\gamma_rdW_r)(\int_{t^m_{k}}^{t}\gamma_rdW_r)^T\rangle-\langle A,\int_{t^m_{k}}^{t}\gamma_r\gamma_r^Tdr\rangle$ is a $P$-martingale on $[t^m_k, t^m_{k+1}]$, then we have
	\begin{align*}
	I_1& \leq E_{P^0}[\sum_{k=0}^{m-1}(\langle A,(\int_{t^m_{k}}^{t^m_{k+1}}\gamma_rdW_r)(\int_{t^m_{k}}^{t^m_{k+1}}\gamma_rdW_r)^T\rangle-\langle A,\int_{t^m_{k}}^{t^m_{k+1}}\gamma_r\gamma_r^Tds\rangle)^2]^\frac12\\
	&\leq \{\sum_{k=0}^{m-1} (2E_{P^0}[\langle A,(\int_{t^m_{k}}^{t^m_{k+1}}\gamma_rdW_r)(\int_{t^m_{k}}^{t^m_{k+1}}\gamma_rdW_r)^T\rangle^2]+ 2E_{P^0}[\langle A,\int_{t^m_{k}}^{t^m_{k+1}}\gamma_r\gamma_r^Tdr\rangle^2])\}^\frac12\\
	&\leq \frac{L_3|A|}{\sqrt{m}}(t-s),
	\end{align*}
where we have used  B-D-G inequality in the last inequality and $L_3$ is a constant depending on  $\Gamma$.
	
By the above estimates and the definition of classical quadratic variation process, we have
$$E_P[(\langle A,\langle B\rangle ^P_t-\langle B\rangle^P_s\rangle-G_1(A)(t-s))^+]=\lim_{m\rightarrow \infty}E_P[(\sum_{k=0}^{m-1}\langle A,(B_{t^m_{k+1}}-B_{t^m_{k}})(B_{t^m_{k+1}}-B_{t^m_{k}})^T\rangle-G_1(A)(t-s))^+]=0,$$
from which, and $\langle B\rangle ^P=\langle M\rangle ^P$, we deduce that
	$\langle A,\langle M\rangle^P_t-\langle M\rangle^P_s\rangle \leq G_1(A)(t-s),  P\text{-a.s.}$
Consequently, we get that
	\begin{equation*}
	\langle A,\frac{d\langle M\rangle^P_t}{dt}\rangle\leq G_1(A),\ \ \ \ \ {a.e.\  t},\  P\text{-a.s.}
	\end{equation*}
The proof is complete.
\end{proof}

\end{document}